\documentclass[dvips,pdf,11pt,leqno]{amsart}
\usepackage{amsmath}
\usepackage{amssymb,latexsym}
\usepackage{graphicx}
\usepackage{epstopdf}
\usepackage{multicol}
\usepackage[utf8x]{inputenc}
\usepackage[T1]{fontenc}
\usepackage{texdraw}

\newtheorem{thm}{\bf Theorem}[section]
\newtheorem{lem}[thm]{\bf Lemma}
\newtheorem{prop}[thm]{\bf Proposition}

\theoremstyle{definition}
\newtheorem{defn}[thm]{\bf Definition}
\newtheorem{ex}[thm]{\bf Example}

\newtheorem*{rem*}{\bf Remark}

\numberwithin{equation}{section}

\newcommand{\BB}{B}
\newcommand{\BBad}{\BB_\text{ad}}
\newcommand{\BBaff}[1][aff]{\BB^\text{#1}}
\newcommand{\C}{\mathcal{C}}
\newcommand{\PP}{\mathcal{P}}
\newcommand{\PPaff}[1][aff]{\PP^\text{#1}}
\newcommand{\g}{\mathfrak{g}}
\newcommand{\h}{\mathfrak{h}}
\newcommand{\Q}{\mathbf{Q}}
\newcommand{\R}{\mathcal{R}}
\newcommand{\Y}{\mathcal{Y}}
\newcommand{\Z}{\mathbf{Z}}
\newcommand{\veps}{\varepsilon}
\newcommand{\vphi}{\varphi}
\DeclareMathOperator{\wt}{wt}
\newcommand{\owt}{\overline{\wt}}

\newcommand{\hwc}{highest weight crystal}
\newcommand{\hwcs}{highest weight crystals}
\newcommand{\Yw}{Young wall}
\newcommand{\Yws}{Young walls}
\newcommand{\rYw}{reduced Young wall}
\newcommand{\rYws}{reduced Young walls}
\newcommand{\Ucry}[1][\g]{$U_q(#1)$-crystal}
\newcommand{\Ucrys}[1][\g]{$U_q(#1)$-crystals}
\newcommand{\Upcry}[1][\g]{$U_q^{'}(#1)$-crystal}
\newcommand{\Upcrys}[1][\g]{$U_q^{'}(#1)$-crystals}
\newcommand{\UAcry}{\Ucry[A_2^{(2)}]}
\newcommand{\UAcrys}{\Ucrys[A_2^{(2)}]}
\newcommand{\UApcry}{\Upcry[A_2^{(2)}]}
\newcommand{\UApcrys}{\Upcrys[A_2^{(2)}]}
\newcommand{\etal}{\textit{et al}.}

\newcommand{\bop}{\bigoplus}
\newcommand{\op}{\oplus}
\newcommand{\ot}{\otimes}

\newcommand{\ltriz}{ \bsegment \move(0 0)\lvec(10 0)\lvec(10 10)\lvec(0
0) \htext(6 2){\tiny{0}} \esegment \rmove(0 10)}
\newcommand{\ltrio}{ \bsegment \move(0 0)\lvec(10 0)\lvec(10 10)\lvec(0
0) \htext(6 2){\tiny{1}} \esegment \rmove(0 10)}
\newcommand{\utriz}{ \bsegment \move(0 0)\lvec(0 10)\lvec(10 10)\lvec(0
0) \htext(2 6){\tiny{0}} \move(0 10) \esegment \rmove(0 10)}
\newcommand{\utrio}{ \bsegment \move(0 0)\lvec(0 10)\lvec(10 10)\lvec(0
0) \htext(2 6){\tiny{1}} \esegment \rmove(0 10)}

\newcommand{\recz}{ \bsegment \move(0
0)\lvec(10 0)\lvec(10 5)\lvec(0 5)\lvec(0 0) \htext(4 1){\tiny{0}} \esegment \rmove(0 5)}
\newcommand{\reco}{ \bsegment \move(0
0)\lvec(10 0)\lvec(10 5)\lvec(0 5)\lvec(0 0) \htext(4 1){\tiny{1}} \esegment \rmove(0 5)}
\newcommand{\recd}{ \bsegment \move(0
0)\lvec(10 0)\lvec(10 10)\lvec(0 10)\lvec(0 0) \vtext(5 2){...} \esegment \rmove(0 10)}

\newcommand{\grecz}{ \bsegment \move(0
0)\lvec(10 0)\lvec(10 5)\lvec(0 5)\lvec(0 0) \lfill f:0.8 \htext(4 1){\tiny{0}} \esegment \rmove(0 5)}

\newcommand{\vpic}[1][0.5em]{\vskip #1}
\newcommand{\hpic}[1][3]{\hskip #1mm}

\newcommand{\abs}[1]{\left\vert#1\right\vert}
\newcommand{\set}[1]{\left\{#1\right\}}

\keywords{quantum affine algebra, Young wall, perfect crystal,
adjoint crystal}

\subjclass[2010] {17B37, 17B67, 20G42, 81R50}

\begin{document}

\title[Young wall model for $A_2^{(2)}$-type adjoint crystals]{Young wall model for $A_2^{(2)}$-type \\ adjoint crystals}%
\author{Seok-Jin Kang}
\address{Joeun Mathematical Research Institute,
441 Yeoksam-ro, Gangnam-gu, Seoul 06196, Korea}%
\email{soccerkang@hotmail.com}


\begin{abstract}

We construct a \Yw\ model for higher level $A_2^{(2)}$-type adjoint
crystals. The \Yws\ and \rYws\ are defined in connection with affine
energy function. We prove that the affine crystal consisting of
\rYws\ provides a realization of \hwcs\ $B(\lambda)$ and
$B(\infty)$.
\end{abstract}
\maketitle

\section*{Introduction}

\vskip 3mm

This paper is a twisted version of \cite{Jung}, where we constructed
a \Yw\ model for $A_1^{(1)}$-type adjoint crystals. Let $\g$ be an
affine Kac-Moody algebra containing a canonical finite dimensional
simple Lie algebra $\g_0$. The \textit{adjoint crystal} is a \Upcry\
of the form
\begin{equation*}
    B = B(0)\op B(\theta) \op \cdots \op B(l \theta) ~~ (l \geq 0),
\end{equation*}
where $\theta$ is the (shortest) maximal root of $\g_0$. In
\cite{BFKL}, Benkart \etal\ introduced the notion of adjoint
crystals and proved that $B = B(0) \op B(\theta)$ is a perfect
crystal of level 1 for \textit{all} quantum affine algebras
(including exceptional types).

The \textit{perfect crystals}, whose theory was developed by Kang
\etal\ \cite{KMN1, KMN2}, provide a representation theoretic
foundation for the theory of vertex models in mathematical physics.
In particular, every perfect crystal yields a \textit{path
realization} of the \hwcs\ $B(\lambda)$ and $B(\infty)$ for any
dominant integral weight $\lambda \in P^+$.

Thus it is an important and interesting problem to find an explicit
construction of perfect crystals. In \cite{KMN1, KMN2, KKM}, it was
shown that the higher level adjoint crystals are perfect for the
quantum affine algebras of type $A_n^{(1)}$, $C_n^{(1)}$,
$A_{2n}^{(2)}$, $D_{n+1}^{(2)}$ with explicit description of
0-arrows. In \cite{FOS, SS}, the adjoint crystals are proven to be
perfect for the affine types $A_{2n-1}^{(2)}$, $B_n^{(1)}$,
$D_n^{(1)}$. Still, it is an open problem to prove the same
statement holds for exceptional affine types.

The \textit{\Yws}, introduced by Kang in \cite{Kang03}, provide a
combinatorial model for level 1 \hwcs\ over quantum affine algebras
of classical type. They consist of colored blocks of various shapes
on a given ground-state wall following the pattern prescribed for
each affine type. The actions of Kashiwara operators are given
explicitly by combinatorics of \Yws\ and the level 1 \hwcs\ are
characterized as affine crystals consisting of \textit{\rYws}.
Moreover, the \textit{characters} can be computed easily by counting
the number of colored blocks in each \rYw. For higher level \hwcs,
we use multi-layer level 1 \Yws\ to obtain a realization \cite{KL}.

All the \Yws\ mentioned above arise from the perfect crystals of
fundamental representations $V(\omega_1)$. In some cases, they
coincide with level 1 adjoint crystals. In \cite{Jung}, Jung \etal\
constructed a \Yw\ model for higher level $A_1^{(1)}$-type adjoint
crystals and proved that the \Ucry[A_1^{(1)}] consisting of \rYws\
provides a realization of higher level \hwc\ $B(\lambda)$ and
$B(\infty)$. The key point of this construction lies in that all the
higher level \Yws\ arising from the adjoint crystals are of
thickness $\leq 1$.

In this paper, we extend the results of \cite{Jung} to the case of
twisted quantum affine algebra $U_q(A_2^{(2)})$. We first recall the
basic properties of $U_q(A_2^{(2)})$ and give a detailed account of
affinization of adjoint crystals, (affine) energy function and
(affine) path realization of higher level \hwcs\ $B(\lambda)$. Then
we construct a \Yw\ model for the $A_2^{(2)}$-type adjoint crystals.
The crucial point is that we stack two 1-blocks successively to
define the Kashiwara operators $\tilde{f_1}$ and $\tilde{F_1}$
taking the identification (\ref{eqn: identification}) into account.

The combinatorics of \Yws\ are explained in Section \ref{Sec:
Crystal structure}, where we define the (affine) crystal structure
on the set of \Yws. The \rYws\ will be defined in connection with
affine energy function. Our main theorem states that the
$A_2^{(2)}$-type affine crystal $\R(\lambda)$ consisting of reduced
\Yws\ on $\lambda$ provides a realization of the \hwc\ $B(\lambda)$
over $U_q(A_2^{(2)})$. The final section is devoted to the \Yw\
realization of $B(\infty)$.

We hope our construction will lead to further developments in \Yw\
model theory associated with higher level adjoint crystals (and
general perfect crystals) for all quantum affine algebras.

\vskip 3mm

{\bf Acknowledgments.} The author would like to express his sincere
gratitude to In-Je Lee, S.J. for his great help in writing this
manuscript.

\vskip 5mm

\section{Quantum affine algebras}

\vskip 3mm

Let $I = \set{0, 1, \cdots, n}$ be a finite index set and let $(A,
P, \Pi, P^\vee, \Pi^\vee)$ be an affine Cartan datum consisting of:
\begin{enumerate}
    \item a symmetrizable Cartan matrix $A = (A_{ij})_{i, j \in I}$ of affine type,
    \item a free abelian group $P$ of rank $n+2$,
    \item a $\Q$-linearly independent set $\Pi = \set{\alpha_i \in P ~|~ i \in I}$,
    \item $P^\vee = \text{Hom}_\Z (P, \Z)$,
    \item $\Pi^\vee = \set{h_i \in P^\vee ~|~ i \in I}$
\end{enumerate}
satisfying $\langle h_i, \alpha_j \rangle = a_{ij}$ $(i, j \in I)$.

Set $\h = \Q \ot_\Z P^\vee$. Take an element $d \in P^\vee$ such
that $\alpha_i(d) = \delta_{i,0}$ $(i \in I)$. The
\textit{fundamental weights} $\Lambda_i \in \h^\ast$ $(i \in I)$ are
defined by
\begin{equation*}
    \begin{array}{ll}
        \Lambda_i (h_j) = \delta_{ij}, & \Lambda_i(d) = 0 \ \ \text{for $i, j \in I$}.
    \end{array}
\end{equation*}

There is a non-degenerate symmetric bilinear form $(\ ,\ )$ on $\h^\ast$ satisfying
\begin{equation*}
    \langle h_i, \lambda \rangle = \frac{2(\alpha_i, \lambda)}{(\alpha_i, \alpha_i)} \ \ \text{for $\lambda \in P$, $i \in I$}.
\end{equation*}
(See \cite{Kac90} for more details.)

Let $q$ be an indeterminate and set
\begin{equation*}
    \begin{array}{ll}
        q_i = q^{(\alpha_i, \alpha_i)/2}, & [n]_i = \displaystyle \frac{q_i^n - q_i^{-n}}{q_i - q_i^{-1}}.
    \end{array}
\end{equation*}

\begin{defn}
    The \textit{quantum affine algebra} $U_q(\g)$ associated with $(A, P, \Pi, $ $P^\vee, \Pi^\vee)$
    is the $\Q(q)$-algebra with 1 generated by $e_i$, $f_i$ $(i \in I)$, $q^h$ $(h \in P^\vee)$ subject to the defining relations
    \begin{equation}
        \begin{aligned}
            & q^0 = 1, \ q^h q^{h'} = q^{h+h'} \ (h, h' \in P^\vee), \\
            & q^h e_j q^{-h} = q^{\langle h, \alpha_j \rangle} e_j, \\
            & q^h f_j q^{-h} = q^{-\langle h, \alpha_j \rangle} f_j \ \text{$(h \in P^\vee$, $j \in I)$}, \\
            & e_i f_j - f_j e_i = \delta_{ij} \frac{K_i - K_i^{-1}}{q_i - q_i^{-1}}, \\
            & \sum_{k=0}^{1 - a_{ij}} (-1)^k e_i^{(1 - a_{ij} - k)} e_j e_i^{(k)} = 0 \ (i \neq j), \\
            & \sum_{k=0}^{1 - a_{ij}} (-1)^k f_i^{(1 - a_{ij} - k)} f_j f_i^{(k)} = 0 \ (i \neq j),
        \end{aligned}
    \end{equation}
    where $K_i = q^{\frac{(\alpha_i, \alpha_i)}{2} h_i}$, $e_i^{(k)} = e_i^k / [k]_i !$, $f_i^{(k)} = f_i^k / [k]_i !$.
\end{defn}

The subalgebra $U^{'}_q (\g)$ generated by $e_i, f_i, K_i^{\pm 1}$ $(i \in I)$ is also called the \textit{quantum affine algebra}.

We denote by $c = c_0 h_0 + \cdots + c_n h_n$ the \textit{canonical central element} and $\delta = d_0 \alpha_0 + \cdots + d_n \alpha_n$ the \textit{null root}, where the coefficients $c_i$, $d_i$ $(i \in I)$ are given in \cite{Kac90}. Then we have
\begin{equation*}
    \begin{aligned}
        P       & = \bop_{i \in I} \Z \Lambda_i \op \Z \left( \frac{1}{d_0} \delta \right), \\
        P^\vee  & = \bop_{i \in I} \Z h_i \op \Z d.
    \end{aligned}
\end{equation*}
Note that $d_0 = 1$ except for the affine $A_{2n}^{(2)}$-type, in which case $d_0 = 2$.

Let $\text{cl}: P \to P / \Z(\frac{1}{d_0} \delta)$ be the canonical projection and let
\begin{equation*}
    P_\text{cl} := \text{cl}(P) = \bop_{i \in I} \Z \text{cl}(\Lambda_i).
\end{equation*}
Moreover, we define
\begin{equation*}
    P^+ = \set{\lambda \in P ~|~ \langle h_i, \lambda \rangle \geq 0 \ \ \text{for all $i \in I$}}
\end{equation*}
and set $P_\text{cl}^+ := \text{cl} (P^+)$. The \textit{level} of $\lambda \in P$ is defined to be $\langle c, \lambda \rangle \in \Z$.

\begin{defn}
    A \textit{\Ucry} (resp. \textit{\Upcry}) is a set $\BB$ together with the maps $\tilde{e_i}$, $\tilde{f_i}: \BB \to \BB ~ \cup ~ \set{0}$, $\veps_i$, $\vphi_i: \BB \to \Z ~ \cup ~ \set{-\infty}$, $\wt: \BB \to P$ (resp. $\owt: \BB \to P_\text{cl}$) satisfying the following conditions:
    \begin{enumerate}
        \item for each $i \in I$, we have $\vphi_i(b) = \veps_i(b) + \langle h_i, \wt(b) \rangle$ (resp. $\vphi_i(b) = \veps_i(b) + \langle h_i, \owt(b) \rangle$),
        \item $\wt(\tilde{e_i} b) = \wt (b) + \alpha_i$ (resp.  $\owt(\tilde{e_i} b) = \owt(b) + \alpha_i$) if $e_i b \in \BB$,
        \item $\wt(\tilde{f_i} b) = \wt (b) - \alpha_i$ (resp.  $\owt(\tilde{f_i} b) = \owt(b) - \alpha_i$) if $f_i b \in \BB$,
        \item $\veps_i (\tilde{e_i} b) = \veps_i (b) - 1$, $\vphi_i (\tilde{e_i} b) = \vphi_i (b) + 1$ if $\tilde{e_i} b \in \BB$,
        \item $\veps_i (\tilde{f_i} b) = \veps_i (b) + 1$, $\vphi_i (\tilde{f_i} b) = \vphi_i (b) - 1$ if $\tilde{f_i} b \in \BB$,
        \item $\tilde{f_i} b = b^{'}$ if and only if $b = \tilde{e_i} b^{'}$ for all $b, b^{'} \in \BB$, $i \in I$,
        \item if $\vphi_i (b) = -\infty$, then $\tilde{e_i} b = \tilde{f_i} b = 0$.
    \end{enumerate}
\end{defn}

Let $\BB$ be a \Upcry. For an element $b \in \BB$, we define
\begin{equation}
    \begin{aligned}
        \veps (b)   & = \sum_{i \in I} \veps_i (b) \text{cl} (\Lambda_i), \\
        \vphi (b)   & = \sum_{i \in I} \vphi_i (b) \text{cl} (\Lambda_i).
    \end{aligned}
\end{equation}
For simplicity, we will often write $\Lambda_i$ for $\text{cl}(\Lambda_i)$ $(i \in I)$. Also we will write $\wt (b)$ for $\owt(b)$ when there is no danger of confusion.

Let $\BB_1$, $\BB_2$ be \Ucrys\ or \Upcrys. The \textit{tensor product} $\BB_1 \ot \BB_2 = \BB_1 \times \BB_2$ is defined by
\begin{equation} \label{eq:tensor}
    \begin{aligned}
        \wt (b_1 \ot b_2)           & = \wt (b_1) + \wt (b_2), \\
        \veps_i (b_1 \ot b_2)       & = \max (\veps_i (b_1), \veps_i (b_2) - \langle h_i, \wt (b_1) \rangle), \\
        \vphi_i (b_1 \ot b_2)       & = \max (\vphi_i (b_2), \vphi_i (b_1) + \langle h_i, \wt (b_2) \rangle), \\
        \tilde{e_i} (b_1 \ot b_2)   & = \begin{cases}
            \tilde{e_i} b_1 \ot b_2 & \text{if $\vphi_i (b_1) \geq \veps_i (b_2)$}, \\
            b_1 \ot \tilde{e_i} b_2 & \text{if $\vphi_i (b_1) < \veps_i (b_2)$},
        \end{cases} \\
        \tilde{f_i} (b_1 \ot b_2)   & = \begin{cases}
            \tilde{f_i} b_1 \ot b_2 & \text{if $\vphi_i (b_1) > \veps_i (b_2)$}, \\
            b_1 \ot \tilde{f_i} b_2 & \text{if $\vphi_i (b_1) \leq \veps_i (b_2)$}.
        \end{cases}
    \end{aligned}
\end{equation}

Let $\BB$ be a \Upcry. The \textit{affinization} of $\BB$ is the \Ucry
\begin{equation*}
    \BBaff := \set{b(m) ~|~ b \in \BB, \ m \in \Z},
\end{equation*}
where the \Ucry\ structure is defined by
\begin{equation}
    \begin{aligned}
        \wt (b(m))          & = \owt(b) + m \delta, \\
        \tilde{e_0} (b(m))  & = (\tilde{e_0} b)(m+1), \ \tilde{f_0} (b(m)) = (\tilde{f_0} b)(m-1), \\
        \tilde{e_i} (b(m))  & = (\tilde{e_i} b)(m), \ \tilde{f_i} (b(m)) = (\tilde{f_i} b)(m) \ \ \text{for $i \neq 0$}, \\
        \veps_i (b(m))      & = \veps_i (b), \ \vphi_i (b(m)) = \vphi_i (b) \ \ \text{for $i \in I, m \in \Z$}.
    \end{aligned}
\end{equation}

A \textit{classical energy function} on $\BB$ is a $\Z$-valued function $h: \BB \ot \BB \to \Z$ satisfying
\begin{equation}
    \begin{aligned}
        h(\tilde{e_i} (b_1 \ot b_2))    & = h(b_1 \ot b_2) \ \ \text{if $i \neq 0$}, \\
        h(\tilde{e_0} (b_1 \ot b_2))    & = \begin{cases}
            h(b_1 \ot b_2) + 1  & \text{if $\vphi_0(b_1) \geq \veps_0(b_2)$}, \\
            h(b_1 \ot b_2) - 1  & \text{if $\vphi_0(b_1) < \veps_0(b_2)$}.
        \end{cases}
    \end{aligned}
\end{equation}

We define the \textit{affine energy function} $H: \BBaff \ot \BBaff \to \Z$ associated with $h: \BB \ot \BB \to \Z$ by
\begin{equation}
    H(b_1(m) \ot b_2(n)) = m-n-h(b_1 \ot b_2) \ \ \text{for $b_1, b_2 \in \BB$, $m, n \in \Z$}.
\end{equation}

Note that
\begin{equation} \label{eq:affineE}
    \begin{aligned}
        H(\tilde{e_i}(b_1(m) \ot b_2(n)))   & = H(\tilde{f_i} (b_1(m) \ot b_2(n))) \\
                                            & = H(b_1(m) \ot b_2(n)) \ \ \text{for all $i \in I$}.
    \end{aligned}
\end{equation}
Hence $H$ is constant on each connected component of $\BBaff \ot \BBaff$.

\begin{defn}
    Let $l$ be a positive integer. A finite \Upcry\ $\BB$ is called a \textit{perfect crystal of level $l$} if $\BB$ satisfies the following conditions:
    \begin{enumerate}
        \item there exists a finite dimensional $U_q^{'}(\g)$-module with crystal $\BB$,
        \item $\BB \ot \BB$ is connected,
        \item there is $\lambda \in P_\text{cl}$ such that $\wt(\BB) \subset \lambda - \sum_{i \neq 0} \Z_{\geq 0} ~ \alpha_i$, $\# (\BB_\lambda) = 1$,
        \item for any $b \in \BB$, we have $\langle c, \veps_i(b) \rangle \geq l$,
        \item for each $\lambda \in P_\text{cl}^+$ with $\langle c, \lambda \rangle = l$, there exist uniquely determined elements $b^\lambda$ and $b_\lambda$ such that $\veps(b^\lambda) = \vphi (b_\lambda) = \lambda$.
    \end{enumerate}
\end{defn}

The elements $b^\lambda$ and $b_\lambda$ are called the \textit{minimal vectors}.

For a dominant integral weight $\lambda \in P^+$, we denote by $\BB(\lambda)$ the crsytal of $V(\lambda)$, the irreducible highest weight $U_q(\g)$-module. By the definition of perfect crystals, we obtain a \Upcry\ isomorphism
\begin{equation} \label{eqn: Psi}
    \begin{array}{rccc}
        \Psi_\lambda:   & \BB(\lambda)  & \stackrel{\sim}{\longrightarrow}  & \BB(\veps(b_\lambda)) \ot \BB \\
                        & u_\lambda     & \longmapsto                       & u_{\veps(b_\lambda)} \ot b_\lambda,
    \end{array}
\end{equation}
where $u_\lambda$ denotes the highest weight vector of $\BB(\lambda)$. Applying the isomorphism (\ref{eqn: Psi}) repeatedly, we get a sequence of \Upcry\ isomorphisms
\begin{equation}
    \begin{array}{ccccccl}
        \BB(\lambda) & \stackrel{\sim}{\longrightarrow} & \BB(\lambda_1) \ot \BB & \stackrel{\sim}{\longrightarrow} & \BB(\lambda_2) \ot \BB \ot \BB & \stackrel{\sim}{\longrightarrow} & \cdots \\
        u_\lambda & \longmapsto & u_{\lambda_1} \ot b_0 & \longmapsto & u_{\lambda_2} \ot b_1 \ot b_0 & \longmapsto & \cdots,
    \end{array}
\end{equation}
where
\begin{equation}
    \begin{array}{ll}
        \lambda_0 = \lambda,        & b_0 = b_{\lambda_0} = b_\lambda, \\
        \lambda_{k+1} = \veps(b_k), & b_{k+1} = b_{\lambda_{k+1}} \ \ \text{for $k \geq 0$}.
    \end{array}
\end{equation}
The infinite sequence $\mathbf{b}_\lambda =(b_k)_{k \geq 0} = \cdots \ot b_k \ot \cdots \ot b_1 \ot b_0 \in \BB^{\ot \infty}$ is called the \textit{ground-state path of weight $\lambda$}.

An infinite sequence $\mathbf{p} = (p_k)_{k \geq 0} \in \BB^{\ot \infty}$ is called a \textit{$\lambda$-path} if $p_k = b_k$ for $k \gg 0$.

Let $\PP(\lambda)$ be the set of all $\lambda$-paths. By the tensor
product rule \eqref{eq:tensor}, $\PP(\lambda)$ is given a \Upcry\
structure. It is proved in \cite{KMN1} that $\PP(\lambda)$ provides
a realization of $\BB(\lambda)$ as a \Upcry.

\begin{prop} \cite{KMN1}
    For each $\lambda \in P^+$, there exists a \Upcry\ isomorphism
    \begin{equation}
        \begin{array}{ll}
            \Phi_\lambda: \BB(\lambda) \longrightarrow \PP(\lambda), & u_\lambda \longmapsto \mathbf{b}_\lambda.
        \end{array}
    \end{equation}
\end{prop}

Note that the \Upcry\ isomorphism (\ref{eqn: Psi}) induces a \Ucry\ embedding
\begin{equation}
    \begin{array}{rcl}
        \Phi_\lambda^{\text{aff}}: \BB(\lambda) & \hookrightarrow & \BB(\veps(b_\lambda)) \ot \BBaff,\\
        u_\lambda & \longmapsto & u_{\veps(b_\lambda)} \ot b_\lambda(m) \ \ \text{for some $m \in \Z$}.
    \end{array}
\end{equation}

Since any affine energy function is constant on each connected
component of $\BBaff \ot \BBaff$, the \Ucry\ $\BB(\lambda)$ is
isomorphic to any connected component of $(\BBaff)^{\ot \infty}$
containing an \textit{affine ground-state path of weight $\lambda$}
\begin{equation}
    \mathbf{b}_\lambda^{\text{aff}}(\mathbf{m}) = (\mathbf{b}_\lambda^{\text{aff}}(m_k))_{k \geq 0} = \cdots \ot b_k(m_k) \ot \cdots \ot b_1(m_1) \ot b_0(m_0)
\end{equation}
for some sequence $\mathbf{m} = (m_k)_{k \geq 0}$ of integers such that
\begin{equation} \label{eqn: integer_m}
    m_k = m_0 + \sum_{j=0}^{k-1} h(b_{j+1} \ot b_j),
\end{equation}
where $\mathbf{b}_\lambda = (b_k)_{k \geq 0}$ is the ground-state path of weight $\lambda$ and $h: \BB \ot \BB \to \Z$ is a classical energy function on $\BB$. Therefore, we obtain a \Ucry\ isomorphism
\begin{equation}
    \BB(\lambda) \stackrel{\sim}{\longrightarrow} \PPaff (\lambda, \mathbf{m}),
\end{equation}
where the \Ucry\ $\PPaff (\lambda, \mathbf{m})$ consists of \textit{affine $\lambda$-paths}
\begin{equation*}
    \mathbf{p} = (p_k (n_k))_{k \geq 0} \in (\BBaff)^{\ot \infty}
\end{equation*}
satisfying the conditions
\begin{enumerate}
    \item $\mathbf{m} = (m_k)_{k \geq 0}$ is a sequence of integers satisfying (\ref{eqn: integer_m}),
    \item $p_k (n_k) = b_k (m_k)$ for $k \gg 0$,
    \item $H (p_{k+1} (n_{k+1}) \ot p_k (n_k)) = H (b_{k+1} (m_{k+1}) \ot b_k (m_k))$ for all $k \geq 0$.
\end{enumerate}

\vskip 3mm

\section{$A_2^{(2)}$-type adjoint crystals}

From now on, we focus on the quantum affine algebra $U_q(A_2^{(2)})$
associated with the affine Cartan matrix $A = \left( \begin{matrix}
2 & -4 \\ -1 & 2 \end{matrix} \right)$. In this case, we have $c =
h_0 + 2h_1$, $\delta = 2 \alpha_0 + \alpha_1$.

Fix a positive integer $l$, and set
\begin{equation}
    \BBad = \set{(x,y) \in \Z_{\geq 0} \times \Z_{\geq 0} ~|~ 0 \leq x+y \leq l}.
\end{equation}
Define a \UApcry\ structure on $\BBad$ by
\begin{equation}
    \begin{aligned}
        \wt (x,y)           & = 2(y-x) \Lambda_0 + (x-y) \Lambda_1, \\
        \veps_1 (x,y)       & = y, \ \vphi_1 (x,y) = x, \\
        \veps_0 (x,y)       & = l - 2y + \abs{x-y}, \\
        \vphi_0 (x,y)       & = l - 2x + \abs{x-y}, \\
        \tilde{e_1} (x,y)   & = (x+1, y-1), \\
        \tilde{f_1} (x,y)   & = (x-1, y+1), \\
        \tilde{e_0} (x,y)   & = \begin{cases}
            (x-1, y)    & \text{if $x>y$}, \\
            (x, y+1)    & \text{if $x \leq y$},
        \end{cases}\\
        \tilde{f_0} (x,y)   & = \begin{cases}
            (x+1, y)    & \text{if $x \geq y$}, \\
            (x, y-1)    & \text{if $x<y$},
        \end{cases}
    \end{aligned}
\end{equation}
whenever $\tilde{e_i}, \tilde{f_i}$ $(i=0,1)$ send non-zero vectors
to non-zero vectors. Otherwise, we define $\tilde{e_i} b =
\tilde{f_i} b = 0$ for $b \in \BBad$. Then it was shown in
\cite{KKM} that $\BBad$ is a $A_2^{(2)}$-type perfect crystal of
level $l$ with minimal vectors
\begin{equation}
    b^\lambda = b_\lambda = (a,a) \ \ \text{for $\lambda = (l-2a) \Lambda_0 + a \Lambda_1$}.
\end{equation}
The \UApcry\ $\BBad$ is called the \textit{$A_2^{(2)}$-type adjoint
crystal of level $l$}.

\begin{ex}
    When $l=4$, the description of $\BBad$ is given below.
    \vpic[1.5em]
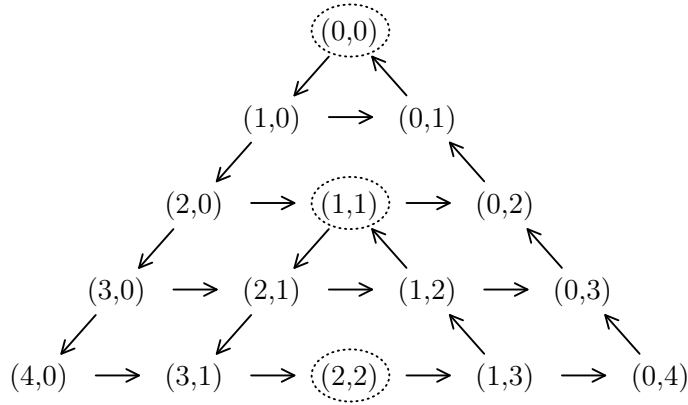
    \begin{center}
        \begin{texdraw}
            \drawdim em \setunitscale 0.15 \linewd 0.4 \arrowheadtype t:V \arrowheadsize l:3 w:3
            \htext(0 2){(4,0)} \move(20 6) \ravec(10 0) \htext(36 2){(3,1)} \move(56 6) \ravec(10 0) \htext(72 2){(2,2)} \move(92 6) \ravec(10 0) \htext(108 2){(1,3)} \move(128 6) \ravec(10 0) \htext(144 2){(0,4)}
            \htext(18 22){(3,0)} \move(38 26) \ravec(10 0) \htext(54 22){(2,1)} \move(74 26) \ravec(10 0) \htext(90 22){(1,2)} \move(110 26) \ravec(10 0) \htext(126 22){(0,3)}
            \htext(36 42){(2,0)} \move(56 46) \ravec(10 0) \htext(72 42){(1,1)} \move(92 46) \ravec(10 0) \htext(108 42){(0,2)}
            \htext(54 62){(1,0)} \move(74 66) \ravec(10 0) \htext(90 62){(0,1)}
            \htext(72 82){(0,0)}

            \move(20 20) \ravec(-8 -9) \move(56 20) \ravec(-8 -9) \move(110 11) \ravec(-8 9) \move(146 11) \ravec(-8 9)
            \move(38 40) \ravec(-8 -9) \move(74 40) \ravec(-8 -9) \move(92 31) \ravec(-8 9) \move(128 31) \ravec(-8 9)
            \move(56 60) \ravec(-8 -9) \move(110 51) \ravec(-8 9)
            \move(74 80) \ravec(-8 -9) \move(92 71) \ravec(-8 9)

            \lpatt(0.3 1) \move(79 6) \lellip rx:9 ry:6 \move(79 46) \lellip rx:9 ry:6 \move(79 86) \lellip rx:9 ry:6
        \end{texdraw}
    \end{center}

    The dotted ones are minimal vectors:
    \begin{equation*}
        \begin{aligned}
            & b^{4 \Lambda_0} = b_{4 \Lambda_0} = (0,0), \\
            & b^{2 \Lambda_0 + \Lambda_1} = b_{2 \Lambda_0 + \Lambda_1} = (1,1), \\
            & b^{2 \Lambda_1} = b_{2 \Lambda_1} = (2,2).
        \end{aligned}
    \end{equation*}
\end{ex}

For $\lambda = (l-2a) \Lambda_0 + a \Lambda_1$ $(0 \leq a \leq \lfloor \frac{l}{2} \rfloor)$, by (\ref{eqn: Psi}), we obtain a \UApcry\ isomorphism
\begin{equation}
    \begin{array}{rccc}
        \Psi_\lambda:   & \BB(\lambda)  & \stackrel{\sim}{\longrightarrow}  & \BB(\lambda) \ot \BBad \\
                        & u_\lambda     & \longmapsto                       & u_\lambda \ot (a,a),
    \end{array}
\end{equation}
which yields a path realization
\begin{equation}
    \BB(\lambda) \stackrel{\sim}{\longrightarrow} \PP(\lambda) = \set{\mathbf{p} = (x_k, y_k)_{k \geq 0} \in \BBad^{\ot \infty} ~|~ (x_k, y_k) = (a,a) \ \ \text{for $k \gg 0$}}.
\end{equation}

Define a function $h: \BBad \ot \BBad \to \Z$ by
\begin{equation} \label{eq:h1}
\begin{aligned}
& h((x_1, y_1) \otimes (x_2, y_2)) \\
& = \text{max} \left \{
\begin{aligned}
& (x_1 + y_1) -(x_2 + y_2),
\ (x_2 + y_2) - (x_1 + y_1), \\
& (x_2 + y_2) + (y_1 - 3 x_1), \
 (x_1 + y_1) + (x_2 - 3 y_2)
\end{aligned}
\right \}.
\end{aligned}
\end{equation}

\vskip 3mm

\begin{prop} \cite{KKM}
    The function $h: \BB_\textnormal{ad} \ot \BB_\textnormal{ad} \rightarrow \Z$ given by {\rm (\ref{eq:h1})} is a classical energy function on
    $\BB_\textnormal{ad}$. In particular, $h((a,a) \otimes
    (a,a))=0$.
\end{prop}

\vskip 3mm

For $\lambda = (l-2a) \Lambda_0 + a \Lambda_1$, let
\begin{equation*}
    \mathbf{b}_\lambda^\text{aff} = (b_k (m_k))_{k \geq 0} = ((a,a) (m_k))_{k \geq 0}
\end{equation*}
be the affine ground-state path with $m_0 = 0$. Then, by (\ref{eqn: integer_m}), $m_k = 0$ for all $k \geq 0$. Hence the \UAcry\ $\PPaff (\lambda)$ consists of affine $\lambda$-paths $\mathbf{p} = ((x_k, y_k) (-m_k))$ with $x_k, y_k, m_k \in \Z_{\geq 0}$ such that
\begin{enumerate}
    \item $(x_k, y_k)(-m_k) = (a,a)(0)$ for $k \gg 0$,
    \item $H((x_{k+1}, y_{k+1})(-m_{k+1}) \ot (x_k, y_k)(-m_k)) = 0$ for all $k \geq 0$.
\end{enumerate}
Note that the condition (ii) is equivalent to

\begin{equation} \label{eq:(ii)}
m_{k} -m_{k+1} = \text{max} \left \{
\begin{aligned}
& (x_{k+1} + y_{k+1})-(x_{k}+y_{k}), \
 (x_{k} + y_{k})-(x_{k+1}+y_{k+1}), \\
& (x_{k}+y_{k}) + (y_{k+1} - 3 x_{k+1}), \
 (x_{k+1} + y_{k+1}) + (x_{k}-3y_{k})
\end{aligned}
\right \}.
\end{equation}

\vskip 5mm

\section{Young wall model for $\BB_\textnormal{ad}$}

\vskip 3mm

In this section, we construct a \Yw\ model associated with $\BBad$.
The colored blocks of the following shapes will be used for this
\Yw\ model.

\vskip 3mm

\vpic
\raisebox{-0.4\height}{
    \begin{texdraw}
        \drawdim em \setunitscale 0.15  \linewd 0.4 \move(0 0)\lvec(10 0)\lvec(10 10)\lvec(0 10)\lvec(0 0) \move(10 0)\lvec(12.5 2.5)\lvec(12.5 12.5)\lvec(2.5 12.5)\lvec(0 10) \move(10 10)\lvec(12.5 12.5)  \htext(3.5 3){0}
    \end{texdraw}} \hskip 1mm , \hskip 1mm
\raisebox{-0.4\height}{
    \begin{texdraw}
        \drawdim em \setunitscale 0.15  \linewd 0.4 \move(0 0)\lvec(10 0)\lvec(10 10)\lvec(0 10)\lvec(0 0) \move(10 0)\lvec(12.5 2.5)\lvec(12.5 12.5)\lvec(2.5 12.5)\lvec(0 10) \move(10 10)\lvec(12.5 12.5)  \htext(3.5 3){1}
    \end{texdraw}}
\hskip 3.8mm : unit width, unit height, half-unit thickness,

\vskip 3mm

\vpic
\raisebox{-0.4\height}{
    \begin{texdraw}
        \drawdim em \setunitscale 0.15  \linewd 0.4 \move(0 0)\lvec(10 0)\lvec(10 5)\lvec(0 5)\lvec(0 0) \move(0 5)\lvec(5 10)\lvec(15 10)\lvec(10 5) \move(15 10)\lvec(15 5)\lvec(10 0)  \htext(4 1){\tiny{0}}
    \end{texdraw}} \hskip 1mm , \hskip 1mm
\raisebox{-0.4\height}{
    \begin{texdraw}
        \drawdim em \setunitscale 0.15  \linewd 0.4 \move(0 0)\lvec(10 0)\lvec(10 5)\lvec(0 5)\lvec(0 0) \move(0 5)\lvec(5 10)\lvec(15 10)\lvec(10 5) \move(15 10)\lvec(15 5)\lvec(10 0)  \htext(4 1){\tiny{1}}
    \end{texdraw}}
\hskip 3.8mm : unit width, half-unit height, unit thickness.
\vpic

\vskip 3mm

We will use the following notations (see, for instance, \cite{HK, Kang03}).
\begin{center}
    \raisebox{-0.4\height}{
        \begin{texdraw}
            \drawdim em \setunitscale 0.15 \linewd 0.4 \move(0 0)\lvec(10 0)\lvec(10 10)\lvec(0 10)\lvec(0 0)\lvec(10 10) \htext(1.5 3.5){*}
        \end{texdraw}
    } $=$ \raisebox{-0.4\height}{
        \begin{texdraw}
            \drawdim em \setunitscale 0.15 \linewd 0.4 \move(0 0)\lvec(10 0)\lvec(10 10)\lvec(0 10)\lvec(0 0) \move(0 10)\lvec(2.5 12.5)\lvec(12.5 12.5) \lvec(10 10) \move(10 0)\lvec(12.5 2.5)\lvec(12.5 12.5) \lpatt(0.3 1) \move(12.5 2.5)\lvec(15 5)\lvec(12.5 5) \htext(3.5 1.5){*}
        \end{texdraw}} \hskip 1mm , \hpic \raisebox{-0.4\height}{
        \begin{texdraw}
            \drawdim em \setunitscale 0.15 \linewd 0.4 \setgray 1 \move(12.5 2.5)\lvec(10 0)\lvec(0 0)\lvec(2.5 2.5) \setgray 0 \move(2.5 2.5)\lvec(12.5 2.5)\lvec(12.5 12.5)\lvec(2.5 12.5)\lvec(2.5 2.5)\lvec(12.5 12.5) \htext(8 2){*}
        \end{texdraw}
    } $=$ \raisebox{-0.4\height}{
        \begin{texdraw}
            \drawdim em \setunitscale 0.15 \linewd 0.4 \move(2.5 2.5)\lvec(12.5 2.5)\lvec(12.5 12.5)\lvec(2.5 12.5)\lvec(2.5 2.5) \move(12.5 2.5)\lvec(15 5)\lvec(15 15)\lvec(5 15)\lvec(2.5 12.5) \move(12.5 12.5)\lvec(15 15) \lpatt(0.3 1) \move(12.5 2.5)\lvec(10 0)\lvec(0 0)\lvec(2.5 2.5) \htext(6 4){*}
        \end{texdraw}} \hskip 1mm .
\end{center}

\vskip 3mm

We now explain how to construct our \Yw\ model for $\BBad$.

\begin{enumerate}
    \item[(1)] The colored blocks should be stacked in the following pattern.
    
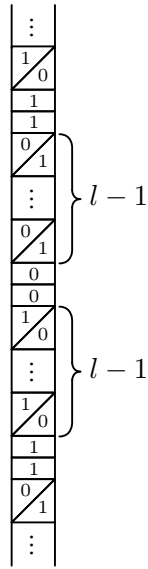
\begin{center}
        \begin{texdraw}
            \drawdim em \setunitscale 0.15 \linewd 0.4
            \move(0 10) \utriz \reco \reco \utrio \recd \utrio \recz \recz \utriz \recd \utriz \reco \reco \utrio
            \move(0 10) \ltrio \rmove(0 10) \ltriz \rmove(0 10) \ltriz \rmove(0 10) \ltrio \rmove(0 10) \ltrio \rmove(0 10) \ltriz
            \move(0 10) \bsegment \lvec(0 -10) \rmove(10 10) \rlvec(0 -10) \vtext(5, -8){...} \esegment
            \move(0 120) \bsegment \lvec(0 10) \rmove(10 -10) \rlvec(0 10) \vtext(5, 2){...} \esegment
            \move(11 30) \clvec(14 30)(14 33)(14 33) \rlvec(0 10)
            \move(16 45) \clvec(14 45)(14 47)(14 47)
            \move(16 45) \clvec(14 45)(14 43)(14 43)
            \move(11 60) \clvec(14 60)(14 57)(14 57) \rlvec (0 -10)
            \htext(18 43){$l-1$}
            \move(11 70) \clvec(14 70)(14 73)(14 73) \rlvec(0 10)
            \move(16 85) \clvec(14 85)(14 87)(14 87)
            \move(16 85) \clvec(14 85)(14 83)(14 83)
            \move(11 100) \clvec(14 100)(14 97)(14 97) \rlvec (0 -10)
            \htext(18 83){$l-1$}
        \end{texdraw}
    \end{center}
    Here $l$ is the level of $\BBad$.

\vskip 3mm

    \item[(2)] We identify the following columns $(1 \leq k \leq l-1)$, which will play a crucial role in our construction.
        \begin{center}
            \begin{equation} \label{eqn: identification}
                \raisebox{-0.5\height}{
                    \begin{texdraw}
                        \drawdim em \setunitscale 0.15 \linewd 0.4
                        \move(0 10) \utrio \recz \recz \utriz \recd \utriz
                        \move(0 10) \ltriz \rmove(0 10) \ltrio \rmove(0 10) \ltrio \ltrio
                        \move(0 10) \bsegment \lvec(0 -10) \rmove(10 10) \rlvec(0 -10) \vtext(5, -8){...} \esegment
                        \move(11 30) \clvec(14 30)(14 33)(14 33) \rlvec(0 15)
                        \move(16 50) \clvec(14 50)(14 52)(14 52)
                        \move(16 50) \clvec(14 50)(14 48)(14 48)
                        \move(11 70) \clvec(14 70)(14 67)(14 67) \rlvec (0 -15)
                        \htext(18 48){$k$} \htext(26 30){=}
                        \move(40 10) \utriz \reco \reco \utrio \recd \utrio
                        \move(40 10) \ltrio \rmove(0 10) \ltriz \rmove(0 10) \ltriz \ltriz
                        \move(40 10) \bsegment \lvec(0 -10) \rmove(10 10) \rlvec(0 -10) \vtext(5, -8){...} \esegment
                        \move(51 30) \clvec(54 30)(54 33)(54 33) \rlvec(0 15)
                        \move(56 50) \clvec(54 50)(54 52)(54 52)
                        \move(56 50) \clvec(54 50)(54 48)(54 48)
                        \move(51 70) \clvec(54 70)(54 67)(54 67) \rlvec (0 -15)
                        \htext(58 48){$l-k$}
                    \end{texdraw}}
            \end{equation}
        \end{center}
    \item[(3)] No blocks of unit thickness can be stacked on top of blocks of half-unit thickness.
    \item[(4)] We stack 0-blocks one-by-one and we stack two 1-blocks successively taking the identification (\ref{eqn: identification}) into account.
    \item[(5)] Under the identification (\ref{eqn: identification}), we can stack the blocks of half-unit thickness in the front repeatedly, but \textit{not} in the back more than once. Hence a column thus obtained has one of the following forms.
        \begin{center}
            \begin{equation}
                \raisebox{-0.5\height}{
                    \begin{texdraw}
                        \drawdim em \setunitscale 0.15 \linewd 0.4
                        \move(0 10) \utrio \recz \recz \utriz \recd \utriz \utriz \move(0 70) \rlvec(0 10) \vtext(5 72) {...} \move(0 80) \utriz
                        \move(0 10) \ltriz \rmove(0 10) \ltrio \rmove(0 10) \ltrio
                        \move(0 10) \bsegment \lvec(0 -10) \rmove(10 10) \rlvec(0 -10) \vtext(5, -8){...} \esegment
                    \end{texdraw} \hskip 6em
                    \begin{texdraw}
                        \drawdim em \setunitscale 0.15 \linewd 0.4
                        \move(0 10) \utriz \reco \reco \utrio \recd \utrio \utrio \move(0 70) \rlvec(0 10) \vtext(5 72) {...} \move(0 80) \utrio
                        \move(0 10) \ltrio \rmove(0 10) \ltriz \rmove(0 10) \ltriz
                        \move(0 10) \bsegment \lvec(0 -10) \rmove(10 10) \rlvec(0 -10) \vtext(5, -8){...} \esegment
                    \end{texdraw} \hskip 6em
                    \begin{texdraw}
                        \drawdim em \setunitscale 0.15 \linewd 0.4
                        \move(0 10) \utriz \reco \reco \utrio \move(0 40) \rlvec(0 30) \move(10 40) \rlvec(0 30) \vtext(5 52) {...} \move(0 70) \utrio
                        \move(0 10) \ltrio \rmove(0 10) \ltriz \rmove(0 30) \ltriz \ltriz
                        \move(0 10) \bsegment \lvec(0 -10) \rmove(10 10) \rlvec(0 -10) \vtext(5, -8){...} \esegment
                    \end{texdraw} \hskip 6em
                    \begin{texdraw}
                        \drawdim em \setunitscale 0.15 \linewd 0.4
                        \move(0 10) \utrio \recz \recz \utriz \move(0 40) \rlvec(0 30) \move(10 40) \rlvec(0 30) \vtext(5 52) {...} \move(0 70) \utriz
                        \move(0 10) \ltriz \rmove(0 10) \ltrio \rmove(0 30) \ltrio \ltrio
                        \move(0 10) \bsegment \lvec(0 -10) \rmove(10 10) \rlvec(0 -10) \vtext(5, -8){...} \esegment
                    \end{texdraw}}
            \end{equation}
        \end{center}
\end{enumerate}

Let $C$ be a column with one of the above forms. We define
$\tilde{f_0} C$ to be the column obtained by stacking a 0-block on
top of $C$. If there is no place to stack a 0-block, we define
$\tilde{f_0} C = 0$.

On the other hand, we define $\tilde{e_0} C$ to be the column obtained by removing a 0-block from the top of $C$. If there is no 0-block that can be removed, we define $\tilde{e_0} C = 0$.

We now define $\tilde{f_1} C$ to be the column obtained by stacking two 1-blocks successively on top of $C$ taking the identification (\ref{eqn: identification}) into account. If there is no such place, we define $\tilde{f_1} C = 0$.

On the other hand, we define $\tilde{e_1} C$ to be the column obtained by removing two 1-blocks successively from the top of $C$ under the identification (\ref{eqn: identification}). If there is no such place, we define $\tilde{e_1} C = 0$.

Thus, with the operators $\tilde{e_i}, \tilde{f_i}$ $(i=0,1)$, we obtain a \textit{\Yw\ model} for $\BBad$.

\begin{ex}
    The following figure gives our \Yw\ model for the level 4 adjoint crystal $\BBad$.
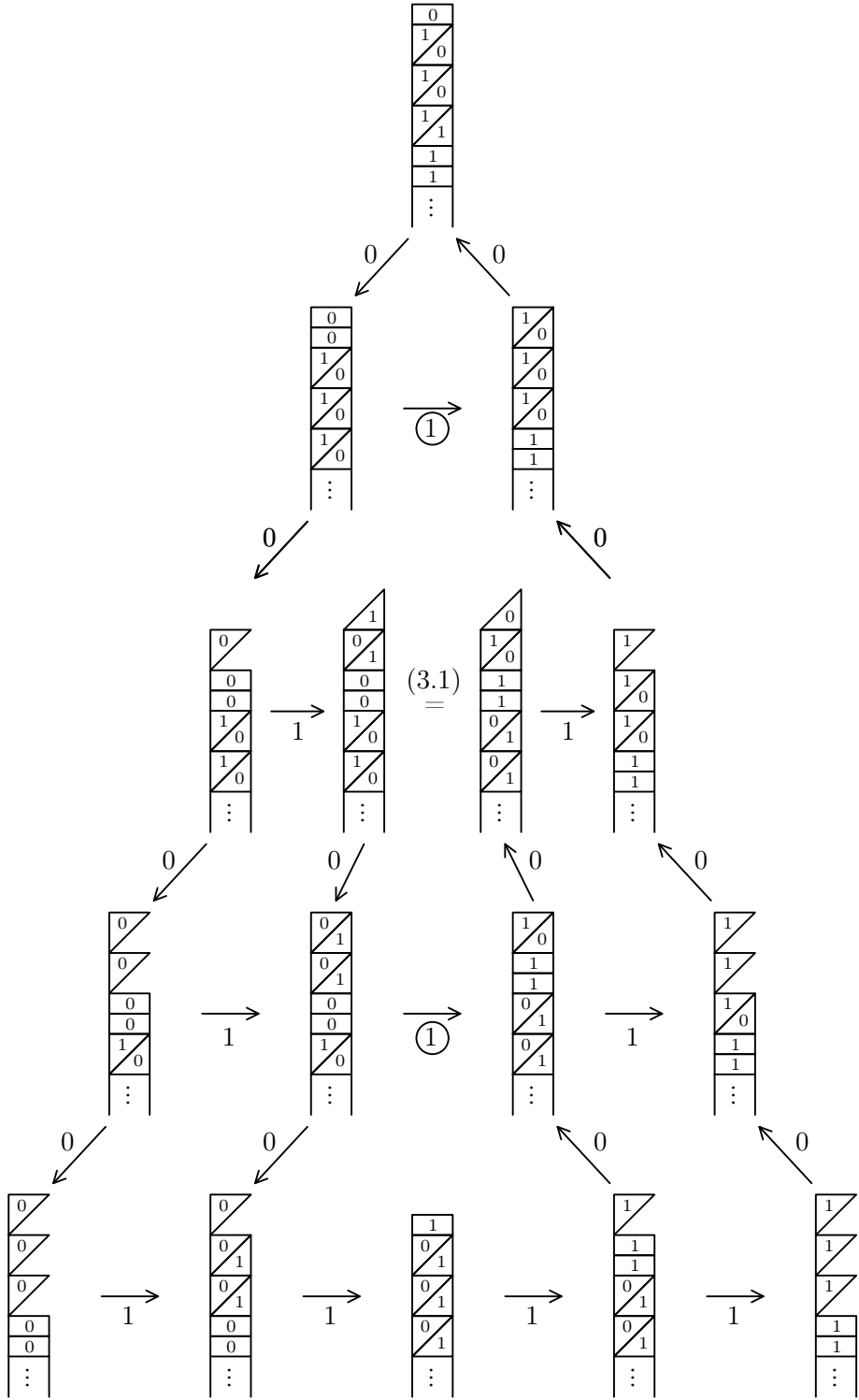
    \begin{center}
        \begin{texdraw}
            \drawdim em \setunitscale 0.15 \linewd 0.4 \arrowheadtype t:V \arrowheadsize l:3 w:3

            \move(0 10) \recz \recz \utriz \utriz \utriz
            \move(0 10) \bsegment \lvec(0 -10) \rmove(10 10) \rlvec(0 -10) \vtext(5, -8){...} \esegment
            \move(23 25) \ravec(14 0) \htext(28 18){1}

            \move(50 10) \recz \recz \utriz \utriz \utriz
            \move(50 20) \ltrio \ltrio
            \move(50 10) \bsegment \lvec(0 -10) \rmove(10 10) \rlvec(0 -10) \vtext(5, -8){...} \esegment
            \move(73 25) \ravec(14 0) \htext(78 18){1}

            \move(100 10) \utriz \utriz \utriz \reco
            \move(100 10) \ltrio \ltrio \ltrio
            \move(100 10) \bsegment \lvec(0 -10) \rmove(10 10) \rlvec(0 -10) \vtext(5, -8){...} \esegment
            \move(123 25) \ravec(14 0) \htext(128 18){1}

            \move(150 10) \utriz \utriz \reco \reco \utrio
            \move(150 10) \ltrio \ltrio
            \move(150 10) \bsegment \lvec(0 -10) \rmove(10 10) \rlvec(0 -10) \vtext(5, -8){...} \esegment
            \move(173 25) \ravec(14 0) \htext(178 18){1}

            \move(200 10) \reco \reco \utrio \utrio \utrio
            \move(200 10) \bsegment \lvec(0 -10) \rmove(10 10) \rlvec(0 -10) \vtext(5, -8){...} \esegment

            \move(24 67) \ravec(-13 -14) \move(74 67) \ravec(-13 -14) \move(149 53) \ravec(-13 14) \move(199 53) \ravec(-13 14)
            \htext(13 61){0} \htext(63 61){0} \htext(145 61){0} \htext(195 61){0}

            \move(25 80) \utrio \recz \recz \utriz \utriz
            \move(25 80) \ltriz
            \move(25 80) \bsegment \lvec(0 -10) \rmove(10 10) \rlvec(0 -10) \vtext(5, -8){...} \esegment
            \move(48 95) \ravec(14 0) \htext(53 87){1}

            \move(75 80) \utrio \recz \recz \utriz \utriz
            \move(75 80) \ltriz \rmove(0 10) \ltrio \ltrio
            \move(75 80) \bsegment \lvec(0 -10) \rmove(10 10) \rlvec(0 -10) \vtext(5, -8){...} \esegment
            \move(98 95) \ravec(14 0) \htext(103 87){1} \move(105 89) \lcir r:4

            \move(125 80) \utriz \utriz \reco \reco \utrio
            \move(125 80) \ltrio \ltrio \rmove(0 10) \ltriz
            \move(125 80) \bsegment \lvec(0 -10) \rmove(10 10) \rlvec(0 -10) \vtext(5, -8){...} \esegment
            \move(148 95) \ravec(14 0) \htext(153 87){1}

            \move(175 80) \reco \reco \utrio \utrio \utrio
            \move(175 90) \ltriz
            \move(175 80) \bsegment \lvec(0 -10) \rmove(10 10) \rlvec(0 -10) \vtext(5, -8){...} \esegment

            \move(49 137) \ravec(-13 -14) \move(88 137) \ravec(-7 -14) \move(130 123) \ravec(-7 14) \move(174 123) \ravec(-13 14)
            \htext(38 131){0} \htext(79 131){0} \htext(129 131){0} \htext(170 131){0}

            \move(74 217) \ravec(-13 -14) \move(149 203) \ravec(-13 14)
            \htext(63 211){0} \htext(145 211){0}

            \move(50 150) \utrio \utrio \recz \recz \utriz
            \move(50 150) \ltriz \ltriz
            \move(50 150) \bsegment \lvec(0 -10) \rmove(10 10) \rlvec(0 -10) \vtext(5, -8){...} \esegment
            \move(65 170) \ravec(13 0) \htext(70 163){1}

            \move(83 150) \utrio \utrio \recz \recz \utriz
            \move(83 150) \ltriz \ltriz \rmove(0 10) \ltrio \ltrio
            \move(83 150) \bsegment \lvec(0 -10) \rmove(10 10) \rlvec(0 -10) \vtext(5, -8){...} \esegment
            \htext(103 170){=} \htext(98.5 174){(\ref{eqn: identification})}

            \move(117 150) \utriz \utriz \reco \reco \utrio
            \move(117 150) \ltrio \ltrio \rmove(0 10) \ltriz \ltriz
            \move(117 150) \bsegment \lvec(0 -10) \rmove(10 10) \rlvec(0 -10) \vtext(5, -8){...} \esegment
            \move(132 170) \ravec(13 0) \htext(137 163){1}

            \move(150 150) \reco \reco \utrio \utrio \utrio
            \move(150 160) \ltriz \ltriz
            \move(150 150) \bsegment \lvec(0 -10) \rmove(10 10) \rlvec(0 -10) \vtext(5, -8){...} \esegment

            \move(74 217) \ravec(-13 -14) \move(149 203) \ravec(-13 14)
            \htext(63 211){0} \htext(145 211){0}

            \move(75 230) \utrio \utrio \utrio \recz \recz
            \move(75 230) \ltriz \ltriz \ltriz
            \move(75 230) \bsegment \lvec(0 -10) \rmove(10 10) \rlvec(0 -10) \vtext(5, -8){...} \esegment
            \move(98 245) \ravec(14 0) \htext(103 238){1} \move(105 240)\lcir r:4

            \move(125 230) \reco \reco \utrio \utrio \utrio
            \move(125 240) \ltriz \ltriz \ltriz
            \move(125 230) \bsegment \lvec(0 -10) \rmove(10 10) \rlvec(0 -10) \vtext(5, -8){...} \esegment

            \move(99 287) \ravec(-13 -14) \move(124 273) \ravec(-13 14)
            \htext(88 281){0} \htext(120 281){0}

            \move(100 300) \reco \reco \utrio \utrio \utrio \recz
            \move(100 310) \ltrio \ltriz \ltriz
            \move(100 300) \bsegment \lvec(0 -10) \rmove(10 10) \rlvec(0 -10) \vtext(5, -8){...} \esegment
        \end{texdraw}
    \end{center}

\vskip 2mm

    Note that we also use the identification (\ref{eqn: identification}) to draw the circled 1-arrows.

\vskip 3mm

    The minimal vectors are
    \begin{equation*}
        \begin{aligned}
            b^{4 \Lambda_0} = b_{4 \Lambda_0} & = \ \raisebox{-0.8\height}{\begin{texdraw}
                    \drawdim em \setunitscale 0.15 \linewd 0.4
                    \move(0 10) \reco \reco \utrio \utrio \utrio \recz
                    \move(0 20) \ltriz \ltriz \ltriz
                    \move(0 10) \bsegment \lvec(0 -10) \rmove(10 10) \rlvec(0 -10) \vtext(5, -8){...} \esegment
                \end{texdraw}} \\ \\
            b^{2 \Lambda_0 + \Lambda_1} = b_{2 \Lambda_0 + \Lambda_1} & = \ \raisebox{-0.8\height}{\begin{texdraw}
                    \drawdim em \setunitscale 0.15 \linewd 0.4
                    \move(0 10) \utrio \utrio \recz \recz \utriz
                    \move(0 10) \ltriz \ltriz \rmove(0 10) \ltrio \ltrio
                    \move(0 10) \bsegment \lvec(0 -10) \rmove(10 10) \rlvec(0 -10) \vtext(5, -8){...} \esegment
                \end{texdraw}} \ = \ \raisebox{-0.8\height}{\begin{texdraw}
                    \drawdim em \setunitscale 0.15 \linewd 0.4
                    \move(0 10) \utriz \utriz \reco \reco \utrio
                    \move(0 10) \ltrio \ltrio \rmove(0 10) \ltriz \ltriz
                    \move(0 10) \bsegment \lvec(0 -10) \rmove(10 10) \rlvec(0 -10) \vtext(5, -8){...} \esegment
                \end{texdraw}} \\ \\
            b^{2 \Lambda_1} = b_{2 \Lambda_1} & = \ \raisebox{-0.8\height}{\begin{texdraw}
                    \drawdim em \setunitscale 0.15 \linewd 0.4
                    \move(0 10) \recz \recz \utriz \utriz \utriz \reco
                    \move(0 20) \ltrio \ltrio \ltrio
                    \move(0 10) \bsegment \lvec(0 -10) \rmove(10 10) \rlvec(0 -10) \vtext(5, -8){...} \esegment
                \end{texdraw}}
        \end{aligned}
    \end{equation*}
\end{ex}

For a dominant integral weight $\lambda = (l-2a) \Lambda_0 + a \Lambda_1$ $(0 \leq a \leq \lfloor \frac{l}{2} \rfloor )$, the \textit{basic ground-state column of weight $\lambda$} is defined as follows:
\begin{enumerate}
    \item $a=0$: \raisebox{-0.8\height}{\begin{texdraw}
            \drawdim em \setunitscale 0.15 \linewd 0.4
            \move(5 10) \reco \utrio \recd \utrio \recz
            \move(5 15) \ltriz \rmove(0 10) \ltriz
            \move(5 10) \bsegment \lvec(0 -10) \rmove(10 10) \rlvec(0 -10) \vtext(5, -8){...} \esegment
            \move(0 50) \rlvec(20 0)
            \move(16 15) \clvec(19 15)(19 18)(19 18) \rlvec(0 10)
            \move(21 30) \clvec(19 30)(19 32)(19 32)
            \move(21 30) \clvec(19 30)(19 28)(19 28)
            \move(16 45) \clvec(19 45)(19 42)(19 42) \rlvec (0 -10)
            \htext(23 28){$l-1$}
        \end{texdraw}} \vskip 5mm
    \item $0 < a < \lfloor \frac{l}{2} \rfloor$: \raisebox{-0.8\height}{\begin{texdraw}
            \drawdim em \setunitscale 0.15 \linewd 0.4
            \move(5 10) \recz \utriz \recd \utriz
            \move(5 15) \ltrio \rmove(0 10) \ltrio \ltrio
            \move(5 10) \bsegment \lvec(0 -10) \rmove(10 10) \rlvec(0 -10) \vtext(5, -8){...} \esegment
            \move(0 50) \rlvec(20 0)
            \move(16 15) \clvec(19 15)(19 18)(19 18) \rlvec(0 15)
            \move(21 35) \clvec(19 35)(19 37)(19 37)
            \move(21 35) \clvec(19 35)(19 33)(19 33)
            \move(16 55) \clvec(19 55)(19 52)(19 52) \rlvec (0 -15)
            \htext(23 33){$2a$}
        \end{texdraw}} \hpic $=$ \hpic \raisebox{-0.8\height}{\begin{texdraw}
            \drawdim em \setunitscale 0.15 \linewd 0.4
            \move(5 10) \reco \utrio \recd \utrio
            \move(5 15) \ltriz \rmove(0 10) \ltriz \ltriz
            \move(5 10) \bsegment \lvec(0 -10) \rmove(10 10) \rlvec(0 -10) \vtext(5, -8){...} \esegment
            \move(0 50) \rlvec(20 0)
            \move(16 15) \clvec(19 15)(19 18)(19 18) \rlvec(0 15)
            \move(21 35) \clvec(19 35)(19 37)(19 37)
            \move(21 35) \clvec(19 35)(19 33)(19 33)
            \move(16 55) \clvec(19 55)(19 52)(19 52) \rlvec (0 -15)
            \htext(23 33){$l-2a$}
        \end{texdraw}} \vskip 5mm
    \item $a = \frac{l}{2}$, $l$ even: \raisebox{-0.8\height}{\begin{texdraw}
            \drawdim em \setunitscale 0.15 \linewd 0.4
            \move(5 10) \recz \utriz \recd \utriz \reco
            \move(5 15) \ltrio \rmove(0 10) \ltrio
            \move(5 10) \bsegment \lvec(0 -10) \rmove(10 10) \rlvec(0 -10) \vtext(5, -8){...} \esegment
            \move(0 50) \rlvec(20 0)
            \move(16 15) \clvec(19 15)(19 18)(19 18) \rlvec(0 10)
            \move(21 30) \clvec(19 30)(19 32)(19 32)
            \move(21 30) \clvec(19 30)(19 28)(19 28)
            \move(16 45) \clvec(19 45)(19 42)(19 42) \rlvec (0 -10)
            \htext(23 28){$l-1$}
        \end{texdraw}} \vskip 5mm
\end{enumerate}
The horizontal lines in the above figure are called the \textit{lines of height $0$}, where we begin to stack blocks.

The following sets of blocks, consisting of two 0-blocks and two 1-blocks, are called the \textit{$\delta$-blocks}.

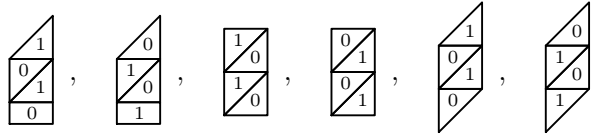
\begin{center}
    \raisebox{-0.4\height}{\begin{texdraw}
        \drawdim em \setunitscale 0.15 \linewd 0.4
        \move(0 0) \recz \utriz \rmove(0 -10) \ltrio \ltrio
    \end{texdraw}} \hskip 1mm, \hskip 4mm \raisebox{-0.4\height}{\begin{texdraw}
        \drawdim em \setunitscale 0.15 \linewd 0.4
        \move(0 0) \reco \utrio \rmove(0 -10) \ltriz \ltriz
    \end{texdraw}} \hskip 1mm, \hskip 4mm \raisebox{-0.4\height}{\begin{texdraw}
        \drawdim em \setunitscale 0.15 \linewd 0.4
        \move(0 0) \utrio \utrio \rmove(0 -20) \ltriz \ltriz
    \end{texdraw}} \hskip 1mm, \hskip 4mm \raisebox{-0.4\height}{\begin{texdraw}
        \drawdim em \setunitscale 0.15 \linewd 0.4
        \move(0 0) \utriz \utriz \rmove(0 -20) \ltrio \ltrio
    \end{texdraw}} \hskip 1mm, \hskip 4mm \raisebox{-0.4\height}{\begin{texdraw}
        \drawdim em \setunitscale 0.15 \linewd 0.4
        \move(0 0) \utriz \utriz \rmove(0 -10) \ltrio \ltrio
    \end{texdraw}} \hskip 1mm, \hskip 4mm \raisebox{-0.4\height}{\begin{texdraw}
        \drawdim em \setunitscale 0.15 \linewd 0.4
        \move(0 0) \utrio \utrio \rmove(0 -10) \ltriz \ltriz
    \end{texdraw}}
\end{center} \vskip 3mm

A column of our \Yw\ model is called a \textit{ground-state column of weight $\lambda$} if it can be obtained from the basic ground-state column of weight $\lambda$ by adding or removing some $\delta$-blocks.

When $\lambda = 2 \Lambda_0 + \Lambda_1$, the basic ground-state column and other ground-state columns are given below.

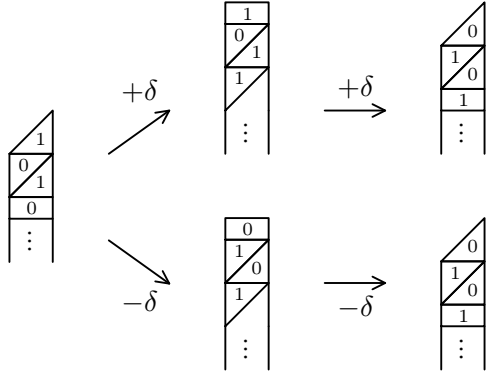
\begin{center}
    \begin{texdraw}
        \drawdim em \setunitscale 0.15 \linewd 0.4 \arrowheadtype t:V \arrowheadsize l:3 w:3
        \move(0 35) \recz \utriz
        \move(0 40) \ltrio \ltrio
        \move(0 35) \bsegment \lvec(0 -10) \rmove(10 10) \rlvec(0 -10) \vtext(5, -8){...} \esegment

        \move(50 10) \utrio \utrio \recz
        \move(60 10) \rlvec(0 10) \rmove(-10 0) \ltriz
        \move(50 10) \bsegment \lvec(0 -10) \rmove(10 10) \rlvec(0 -10) \vtext(5, -8){...} \esegment

        \move(50 60) \utrio \utriz \reco
        \move(60 60) \rlvec(0 10) \rmove(-10 0) \ltrio
        \move(50 60) \bsegment \lvec(0 -10) \rmove(10 10) \rlvec(0 -10) \vtext(5, -8){...} \esegment

        \move(100 10) \reco \utrio
        \move(100 15) \ltriz \ltriz
        \move(100 10) \bsegment \lvec(0 -10) \rmove(10 10) \rlvec(0 -10) \vtext(5, -8){...} \esegment

        \move(100 60) \reco \utrio
        \move(100 65) \ltriz \ltriz
        \move(100 60) \bsegment \lvec(0 -10) \rmove(10 10) \rlvec(0 -10) \vtext(5, -8){...} \esegment

        \move(23 30) \ravec(14 -10) \move(23 50) \ravec(14 10) \move(73 20) \ravec(14 0) \move(73 60) \ravec(14 0)
        \htext(26 62){$+\delta$} \htext(26 13){$-\delta$} \htext(76 63){$+\delta$} \htext(76 12){$-\delta$}
    \end{texdraw}
\end{center}

Let $C$ be a column of our level $l$ \Yw\ model for $\BBad$ and let
$G$ be the column consisting of blocks below the line of height 0.
Then $G$ is a ground-state column. Let $Z$ denote the topmost basic
ground-state column of weight $l \Lambda_0$ contained in $C$. We
define
\begin{equation*}
    \begin{aligned}
        s = s(C)                & = \text{the number of 0-blocks in $C$ above $G$}, \\
        t = t(C)                & = \text{the number of 1-blocks in $C$ above $G$}, \\
        \bar{s} = \bar{s} (C)   & = \text{the number of 0-blocks in $C$ above $Z$}, \\
        \bar{t} = \bar{t} (C)   & = \text{the number of 1-blocks in $C$ above $Z$}.
    \end{aligned}
\end{equation*}

Note that $s - \bar{s} =t - \bar{t}$ and that $t$ and $\bar{t}$ are
even. (Thus $s - \bar{s}$ is also even.) When $s=t$, we have
$\bar{s} = \bar{t}$, but they may have two different values under
the identification (\ref{eqn: identification}). In this case, we
take the smaller value for $\bar{s} = \bar{t}$.

\begin{ex}
    Let $\lambda = 2 \Lambda_0 + \Lambda_1$. \vskip 3mm

    (a) Let $C =$ \raisebox{-0.8\height}{\begin{texdraw}
            \drawdim em \setunitscale 0.15 \linewd 0.4
            \move(5 10) \recz \recz \utriz \utriz \utriz \reco \reco \utrio \utrio \utrio
            \move(5 20) \ltrio \ltrio \ltrio \rmove(0 10) \ltriz
            \move(5 10) \bsegment \lvec(0 -10) \rmove(10 10) \rlvec(0 -10) \vtext(5, -8){...} \esegment
            \move(0 15) \rlvec(20 0) \htext(22 13){$Z$}
            \move(0 35) \rlvec(20 0) \htext(22 33){$G$}
        \end{texdraw}} \hpic . \vskip 3mm
    Then $s=3$, $t=6$, $\bar{s} = 5$, $\bar{t} = 8$. \vskip 3mm

    (b) Let $C =$ \raisebox{-0.8\height}{\begin{texdraw}
            \drawdim em \setunitscale 0.15 \linewd 0.4
            \move(5 10) \reco \reco \utrio \utrio \utrio \recz \recz \utriz \utriz
            \move(5 20) \ltriz \ltriz \ltriz
            \move(5 10) \bsegment \lvec(0 -10) \rmove(10 10) \rlvec(0 -10) \vtext(5, -8){...} \esegment
            \move(0 35) \rlvec(20 0) \htext(22 33){$G$}
            \move(0 55) \rlvec(20 0) \htext(22 53){$Z$}
        \end{texdraw}} \hpic . \vskip 3mm

    Then $s=5$, $t=2$, $\bar{s} = 3$, $\bar{t} = 0$. \vskip 3mm

    (c) Let $C =$ \raisebox{-0.8\height}{\begin{texdraw}
            \drawdim em \setunitscale 0.15 \linewd 0.4
            \move(15 10) \utrio \recz \recz \utriz \utriz \utriz \reco \reco \utrio
            \move(15 10) \ltriz \rmove(0 10) \ltrio \ltrio \ltrio \rmove(0 10) \ltriz \ltriz
            \move(15 10) \bsegment \lvec(0 -10) \rmove(10 10) \rlvec(0 -10) \vtext(5, -8){...} \esegment
            \move(10 25) \rlvec(20 0) \htext(3 23){$Z$}
            \move(10 45) \rlvec(20 0) \htext(3 43){$G$}
        \end{texdraw}} \hskip 3mm $=$ \hskip 3mm \raisebox{-0.8\height}{\begin{texdraw}
            \drawdim em \setunitscale 0.15 \linewd 0.4
            \move(5 10) \utriz \reco \reco \utrio \utrio \utrio \recz \recz \utriz
            \move(5 10) \ltrio \rmove(0 10) \ltriz \ltriz \ltriz \rmove(0 10) \ltrio \ltrio
            \move(5 10) \bsegment \lvec(0 -10) \rmove(10 10) \rlvec(0 -10) \vtext(5, -8){...} \esegment
            \move(0 45) \rlvec(20 0) \htext(22 43){$G$}
            \move(0 65) \rlvec(20 0) \htext(22 63){$Z$}
        \end{texdraw}} \hpic . \vskip 3mm

    Then $s=4$, $t=4$. But under the identification (\ref{eqn: identification}), we have two possibilities: $\bar{s} = \bar{t} = 6$ or $\bar{s} = \bar{t} = 2$. In this case, we take the smaller values $\bar{s} = \bar{t} = 2$.
\end{ex}

Let $C$ be a column of our \Yw\ model for $\BBad$. Then $C$ is
uniquely determined by $s$, $t$, $\bar{s}$ and $\bar{t}$. Since $s -
\bar{s} = t - \bar{t}$, we have only to use $s$, $\bar{s}$ and
$\bar{t}$. Hence we will write
\begin{equation*}
    C = \langle s, t, \bar{s}, \bar{t} \rangle = \langle s, \bar{s}, \bar{t} \rangle.
\end{equation*}
Let $\C$ be the set of all columns of our \Yw\ model for the level
$l$ adjoint crystal $\BBad$ and let
\begin{equation*}
    \BBaff_{\leq 0} = \set{(x,y)(-m) ~|~ x, y, m \in \Z_{\geq 0}}.
\end{equation*}

\vskip 3mm

\begin{prop} \label{prop:Baffine}
{\rm  There is a $U_{q}(A_{2}^{(2)})$-crystal isomorphism $\Psi:\C
\overset{\sim} \rightarrow \BBaff_{\leq 0}$ defined by
\begin{equation} \label{eq:Baffine}
        \Psi (\langle s, \bar{s}, \bar{t} \rangle)  = \begin{cases}
            \left( \bar{s} - \frac{1}{2} \bar{t}, \frac{1}{2} \bar{t} \right) (-s)  & \text{if $\bar{s} \geq \bar{t}$}, \\
            \left( l - \frac{1}{2} \bar{t}, l - \bar{s} + \frac{1}{2} \bar{t} \right) (-s)  & \text{if $\bar{s} <
            \bar{t}$}.
        \end{cases}
\end{equation}

The inverse homomorphism is given by

\begin{equation}
\Phi ((x,y)(-m)) =
\begin{cases}
            \langle m, x+y, 2y \rangle              & \text{if $x \geq y$}, \\
            \langle m, 2l - (x+y), 2(l-x) \rangle   & \text{if $x<y$}.
        \end{cases}
\end{equation}
}
\end{prop}

\begin{proof}
It is easy to see that $\Psi$ and $\Phi$ are inverses to each other.
Moreover, it is straightforward to verify that they commute with
$\tilde{e}_i$ and $\tilde{f}_i$ $(i=0,1)$  whenever all the maps
involved send non-zero vectors to non-zero vectors. For instance, if
$\bar{s} \ge \bar{t}$, we have
$$\Psi(\tilde{f}_{0}(\langle s, \bar{s}, \bar{t} \rangle)) =
\Psi(\langle s+1, \bar{s}+1, \bar{t} \rangle) = \left(\bar{s}+1 -
\frac{1}{2} \bar{t}, \frac{1}{2} \bar{t}\right)(-s-1),$$ whereas
$$\tilde{f}_{0} \Psi(\langle s, \bar{s}, \bar{t} \rangle) =
\tilde{f}_{0} \left(\left(\bar{s}-\frac{1}{2} \bar{t}, \frac{1}{2}
\bar{t}\right)(-s)\right) =\left(\bar{s}-\frac{1}{2} \bar{t} + 1,
\frac{1}{2} \bar{t}\right)(-s-1),$$ as expected.
\end{proof}

Let $C = \langle s, t, \bar{s}, \bar{t} \rangle$ and $C^{'} =
\langle s', t', \bar{s'}, \bar{t'} \rangle$ be columns of our level
$l$ \Yw\ model. Using the isomorphism \eqref{eq:Baffine}, we may
rewrite the affine energy function $H$ on $\BBaff_{\leq 0}$ as
\begin{equation} \label{eq:H}
H(\Psi(C) \otimes \Psi(C')) = -s + s' -h(C,C'), \end{equation}
where
\begin{equation} \label{eq:h1-a}
h(C,C')=\begin{cases}
\text{max}\left(\begin{aligned}& \bar{s}-\bar{s'},
\bar{s}+\bar{s'}-2 \bar{t'}, \\ & \bar{s'}-\bar{s}, \bar{s'}-3
\bar{s} + 2 \bar{t}
\end{aligned}\right) & \  \text{if} \ \
\bar{s} \ge \bar{t}, \ \bar{s'} \ge \bar{t'}, \\
\text{max} \left(\begin{aligned}& \bar{s}+\bar{s'}-2l, 2l- 3
\bar{s}- \bar{s'} + 2 \bar{t},  \\ &  2l-\bar{s}-\bar{s'},
  -2l + \bar{s}+3 \bar{s'} - 2 \bar{t'}\end{aligned}\right) & \
\text{if} \ \ \bar{s} \ge \bar{t}, \ \bar{s'} < \bar{t'},
\\
\text{max} \left(\begin{aligned} & 2l - \bar{s} - \bar{s'}, -2l -
\bar{s} +
\bar{s'} + 2 \bar{t}, \\
& -2l + \bar{s} + \bar{s'}, 2l - \bar{s} + \bar{s'} -2
\bar{t'}\end{aligned} \right) &   \ \text{if} \ \bar{s}< \bar{t},
\bar{s'} \ge \bar{t'}, \\
\text{max} \left(\begin{aligned} & \bar{s} - \bar{s'}, - \bar{s} -
\bar{s'} + 2 \bar{t}, \\ & \bar{s'} - \bar{s}, -\bar{s} + 3 \bar{s'}
-2 \bar{t'} \end{aligned}\right) & \  \text{if} \ \bar{s} < \bar{t},
\ \bar{s'} < \bar{t'}.
\end{cases}
\end{equation}

We will often write $H(C \ot C^{'})$ for $H(\Psi(C) \otimes
\psi(C'))$.

\vskip 3mm

\begin{defn}
    Let $Y = (Y_k)_{k \geq 0}$ be an infinite sequence of columns of our \Yw\ model.

    (a) The sequence $G = (G_k)_{k \geq 0}$ consisting of the blocks below the line of height 0 in $Y$ is called the \textit{ground-state wall} of $Y$.

    (b) $G$ is called the \textit{basic ground-state wall} of $Y$ if $G_k$ is the basic ground-state column for all $k \geq 0$.
\end{defn}

\vskip 3mm

\begin{defn}
   (a)  An infinite sequence $Y = (Y_k)_{k \geq 0}$ of columns of our \Yw\ model is called a {\it Young wall on $\lambda \in P^+$} if
    \begin{enumerate}
       \item $Y_k$ is the basic ground-state column of weight $\lambda$ for $k \gg 0$.
       \item $H(Y_{k+1} \otimes Y_{k}) \ge 0$ for all $k \ge 0$.
    \end{enumerate}

    (b) A Young wall $Y=(Y_{k})_{k \ge 0}$ is called {\it reduced}
    if $H(Y_{k+1} \otimes Y_{k})=0$ for all $k \ge 0$.
\end{defn}

\begin{ex}
    To determine whether an infinite sequence of columns is a (reduced) \Yw\ or not,
    we have only to check the conditions for adjacent columns.
    We will give several examples of the tensor product of two columns when $\lambda = 2 \Lambda_0 + \Lambda_1$.
    The shaded part indicates the basic ground-state column.

    (a) Let $C \ot C^{'} =$ \hpic \raisebox{-0.8\height}{\begin{texdraw}
            \drawdim em \setunitscale 0.15 \linewd 0.4
            \move(0 25) \rlvec(0 25) \rlvec(10 0) \rlvec(0 -5) \rlvec(10 10) \rlvec(0 -25) \rlvec(-10 0) \rlvec(0 5) \rlvec(-10 -10) \lfill f:0.8
            \move(0 10) \recz \utriz \utriz \utriz \reco \reco \utrio
            \move(0 15) \ltrio \ltrio \ltrio
            \move(10 30) \recz \utriz \utriz \utriz \reco
            \move(10 35) \ltrio \ltrio \ltrio
            \move(0 10) \bsegment \lvec(0 -10) \rmove(10 10) \rlvec(0 -10) \vtext(5, -8){...} \esegment
            \move(10 30) \bsegment \lvec(0 -10) \rmove(10 10) \rlvec(0 -10) \vtext(5, -8){...} \esegment
        \end{texdraw}} \ . \vpic
    Then $s=0$, $t=2$, $\bar{s} = 4$, $\bar{t} = 6$ and $s'=2$, $t'=2$, $\bar{s'} = 4$, $\bar{t'} = 4$.
    By \eqref{eq:H} and     \eqref{eq:h1-a}, we have $H(C \ot C^{'}) = -2$. Hence $C \ot
    C^{'}$ does {\it not} satisfy the conditions for Young walls. \vpic

    (b) Let $C \ot C^{'} =$ \hpic \raisebox{-0.8\height}{\begin{texdraw}
            \drawdim em \setunitscale 0.15 \linewd 0.4
            \move(0 25) \rlvec(0 25) \rlvec(10 0) \rlvec(0 -5) \rlvec(10 10) \rlvec(0 -25) \rlvec(-10 0) \rlvec(0 5) \rlvec(-10 -10) \lfill f:0.8
            \move(0 10) \recz \utriz \utriz \utriz \reco \reco \utrio
            \move(0 15) \ltrio \ltrio \ltrio
            \move(10 30) \recz \utriz \utriz \utriz \reco \reco \utrio
            \move(10 35) \ltrio \ltrio \ltrio \rmove(0 10) \ltriz
            \move(0 10) \bsegment \lvec(0 -10) \rmove(10 10) \rlvec(0 -10) \vtext(5, -8){...} \esegment
            \move(10 30) \bsegment \lvec(0 -10) \rmove(10 10) \rlvec(0 -10) \vtext(5, -8){...} \esegment
        \end{texdraw}} \ . \vpic
    Then $s=0$, $t=2$, $\bar{s} = 4$, $\bar{t} = 6$ and $s'=3$, $t'=4$, $\bar{s'} = 5$, $\bar{t'} = 6$.
    By \eqref{eq:H} and \eqref{eq:h1-a}, we have $H(C \ot C^{'})=0$ and hence
 $C \ot C^{'}$ satisfies the conditions for \textit{reduced} Young walls. \vpic

    (c) Let $C \ot C^{'} =$ \hpic \raisebox{-0.8\height}{\begin{texdraw}
            \drawdim em \setunitscale 0.15 \linewd 0.4
            \move(0 25) \rlvec(0 25) \rlvec(10 0) \rlvec(0 -5) \rlvec(10 10) \rlvec(0 -25) \rlvec(-10 0) \rlvec(0 5) \rlvec(-10 -10) \lfill f:0.8
            \move(0 10) \recz \utriz \utriz \utriz \reco \reco \utrio
            \move(0 15) \ltrio \ltrio \ltrio
            \move(10 30) \recz \utriz \utriz \utriz \reco \reco \utrio \utrio \utrio
            \move(10 35) \ltrio \ltrio \ltrio \rmove(0 10) \ltriz \ltriz
            \move(0 10) \bsegment \lvec(0 -10) \rmove(10 10) \rlvec(0 -10) \vtext(5, -8){...} \esegment
            \move(10 30) \bsegment \lvec(0 -10) \rmove(10 10) \rlvec(0 -10) \vtext(5, -8){...} \esegment
        \end{texdraw}} \ . \vpic
    Then $s=0$, $t=2$, $\bar{s} = 4$, $\bar{t} = 6$ and $s'=4$, $t'=6$, $\bar{s'} = 6$, $\bar{t'} = 8$.
    In this case, we have $H(C \ot C^{'}) =2$ and hence $C \ot C^{'}$ satisfies the conditions for
    (\textit{non-reduced}) Young walls.  \vpic

    (d) Let $C \ot C^{'} =$ \hpic \raisebox{-0.8\height}{\begin{texdraw}
            \drawdim em \setunitscale 0.15 \linewd 0.4
            \move(0 25) \rlvec(0 25) \rlvec(10 0) \rlvec(0 -5) \rlvec(10 10) \rlvec(0 -25) \rlvec(-10 0) \rlvec(0 5) \rlvec(-10 -10) \lfill f:0.8
            \move(0 10) \recz \utriz \utriz \utriz \reco \reco \utrio
            \move(0 15) \ltrio \ltrio \ltrio
            \move(10 30) \reco \utrio \utrio \utrio \recz \recz \utriz \utriz
            \move(10 35) \ltriz \ltriz \ltriz \rmove(0 10) \ltrio \ltrio
            \move(0 10) \bsegment \lvec(0 -10) \rmove(10 10) \rlvec(0 -10) \vtext(5, -8){...} \esegment
            \move(10 30) \bsegment \lvec(0 -10) \rmove(10 10) \rlvec(0 -10) \vtext(5, -8){...} \esegment
        \end{texdraw}} \ . \vpic
    Then $s=0$, $t=2$, $\bar{s} = 4$, $\bar{t} = 6$ and $s'=6$, $t'=4$, $\bar{s'} = 4$, $\bar{t'} = 2$.
    In this case, $H(C \ot C^{'})=2$ and hence $C \ot C^{'}$ satisfies the conditions for
    ({\it non-reduced}) Young walls.
\end{ex}

\vskip 5mm

\section{Crystal structure} \label{Sec: Crystal structure}

\vskip 2mm

Let $\Y(\lambda)$ and $\R(\lambda)$ be the set of \Yws\ and reduced
\Yws, respectively. In this section, we describe the combinatorics
of \Yws\ which would provide crystal structures on $\Y(\lambda)$ and
$\R(\lambda)$. To this end, we first define the action of Kashiwara
operators $\tilde{E_i}, \tilde{F_i}$ $(i=0,1)$.

\begin{defn}
    Let $Y = (Y_k)_{k \geq 0}$ be a \Yw.

    (a) A 0-block in $Y$ is called \textit{removable} if it can be removed from $Y$ to get a \Yw.

    (b) A slot in $Y$ is called \textit{$0$-admissible} if one can add a 0-block on $Y$ to obtain a \Yw.

    (c) A pair of 1-blocks in a column of $Y$ is called \textit{removable}
    if these two blocks can be removed successively from $Y$, under the identification (\ref{eqn: identification}), to obtain a \Yw.

    (d) A pair of 1-slots in a column of $Y$ is called \textit{$1$-admissible} if one can add two 1-blocks successively on these slots, under the identification (\ref{eqn: identification}), to obtain a \Yw.
\end{defn}

To each column $Y_k$ of $Y = (Y_k)_{k \geq 0}$, we assign a sequence
of $-$'s and $+$'s from left to right with as many $-$'s as the
number of removable 0-blocks (resp. removable 1-pairs) followed by
as many $+$'s as the number of 0-admissible slots (resp.
1-admissible pairs). This sequence is called the
\textit{$0$-signature} (resp. \textit{$1$-signature}) of $Y_k$. Thus
we get an infinite sequence of $i$-signatures $(i=0,1)$. Cancel out
all $(+,-)$-pairs in this sequence. Then we obtain a finite sequence
of $-$'s followed by $+$'s, which is called the
\textit{$i$-signature} of $Y$ $(i=0,1)$.

We define $\tilde{E_0} Y$ (resp. $\tilde{E_1} Y$) to be the \Yw\ obtained by removing a 0-block (resp. a pair of 1-blocks) from the column of $Y$ that corresponds to the rightmost $-$ in the 0-signature (resp. 1-signature) of $Y$. If there is no $-$ in the $i$-signature of $Y$, we define $\tilde{E_i} Y = 0$ $(i=0,1)$.

On the other hand, we define $\tilde{F_0} Y$ (resp. $\tilde{F_1} Y$) to be the \Yw\ obtained by adding a 0-block (resp. a pair of 1-blocks) to the column of $Y$ that corresponds to the leftmost $+$ in the 0-signature (resp. 1-signature). If there is no $+$ in the $i$-signature of $Y$, we define $\tilde{F_i} Y = 0$ $(i=0,1)$.

The operators $\tilde{E_i}$, $\tilde{F_i}$ $(i=0,1)$ are called the \textit{Kashiwara operators} on \Yws.

\begin{ex} \label{ex:YW}
    Let $\lambda = 2 \Lambda_0 + \Lambda_1$ and let
    \begin{center}
        \begin{equation*}
            \begin{aligned}
                Y & = \hpic \raisebox{-0.5\height}{\begin{texdraw}
                    \drawdim em \setunitscale 0.15 \linewd 0.4
                    \move(20 15) \rlvec(0 15) \rlvec(10 10) \rlvec(0 -10) \rlvec(10 10) \rlvec(0 -10) \rlvec(10 10) \rlvec(0 -5) \rlvec(10 0) \rlvec(0 -5) \rlvec(10 10) \rlvec(0 -25) \rlvec(-10 0) \rlvec(0 5) \rlvec(-10 -10) \rlvec(0 5) \rlvec(-30 0) \lfill f:0.8
                    \htext(5 22){$\cdots$}
                    \move(20 15) \rlvec(0 -15) \rmove(10 15) \rlvec(0 -15) \rmove(10 15) \rlvec(0 -15) \rmove(10 10) \rlvec(0 -10) \rmove(10 15) \rlvec(0 -15) \rmove(10 15) \rlvec(0 -15)
                    \move(20 15) \recz \utriz
                    \rmove(0 -10) \ltrio \ltrio
                    \rmove(10 -25) \recz \utriz
                    \rmove(0 -10) \ltrio \ltrio
                    \rmove(10 -25) \recz \utriz
                    \rmove(0 -10) \ltrio \ltrio
                    \rmove(10 -30) \utriz \utriz \reco \reco \utrio \utrio \utrio
                    \rmove(0 -50) \ltrio \rmove(0 10) \ltriz \ltriz \ltriz
                    \rmove(10 -55) \reco \utrio \utrio \utrio \recz \recz \utriz \utriz \utriz \reco \reco \utrio
                    \rmove(0 -90) \ltriz \ltriz \ltriz \rmove(0 10) \ltrio \ltrio \ltrio \rmove(0 10) \ltriz
                    \vtext(25 2){...} \vtext(35 2){...} \vtext(45 2){...} \vtext(55 2){...} \vtext(65 2){...}
                \end{texdraw}} \\
                & = \hpic \cdots \hpic \raisebox{-0.5\height}{\begin{texdraw}
                    \drawdim em \setunitscale 0.15 \linewd 0.4
                    \move(0 30) \rlvec(0 15) \rlvec(10 10) \rlvec(0 -10) \rlvec(10 10) \rlvec(0 -25) \rlvec(-20 0) \lfill f:0.8
                    \move(0 30) \recz \utriz
                    \rmove(10 -15) \recz \utriz
                    \rmove(-10 -10) \ltrio \ltrio
                    \rmove(10 -20) \ltrio \ltrio
                    \move(0 30) \bsegment \lvec(0 -10) \rmove(10 10) \rlvec(0 -10) \rmove(10 10) \rlvec(0 -10) \vtext(5, -8){...} \vtext(15, -8){...} \esegment
                    \vtext(12 12){$=$} \htext(8 5){$G$}
                \end{texdraw}} \hpic \ot \hpic \raisebox{-0.5\height}{\begin{texdraw}
                    \drawdim em \setunitscale 0.15 \linewd 0.4
                    \move(0 30) \rlvec(0 15) \rlvec(10 10) \rlvec(0 -25) \rlvec(-10 0) \lfill f:0.8
                    \move(0 30) \recz \utriz
                    \rmove(0 -10) \ltrio \ltrio
                    \move(0 30) \bsegment \lvec(0 -10) \rmove(10 10) \rlvec(0 -10) \vtext(5, -8){...} \esegment
                    \vtext(7 12){$=$} \htext(3 3){$Y_2$}
                \end{texdraw}} \hpic \ot \hpic \raisebox{-0.5\height}{\begin{texdraw}
                    \drawdim em \setunitscale 0.15 \linewd 0.4
                    \move(0 30) \rlvec(0 25) \rlvec(10 0) \rlvec(0 -15) \rlvec(-10 -10) \lfill f:0.8
                    \move(0 30) \utriz \utriz \reco \reco \utrio \utrio \utrio
                    \rmove(10 -60) \rlvec(0 10) \rmove(-10 0) \ltrio \rmove(0 10) \ltriz \ltriz \ltriz
                    \move(0 30) \bsegment \lvec(0 -10) \rmove(10 10) \rlvec(0 -10) \vtext(5, -8){...} \esegment
                    \vtext(7 12){$=$} \htext(3 3){$Y_1$}
                \end{texdraw}} \hpic \ot \hpic \raisebox{-0.5\height}{\begin{texdraw}
                    \drawdim em \setunitscale 0.15 \linewd 0.4
                    \move(0 30) \rlvec(0 15) \rlvec(10 10) \rlvec(0 -25) \rlvec(-10 0) \lfill f:0.8
                    \move(0 30) \reco \utrio \utrio \utrio \recz \recz \utriz \utriz \utriz \reco \reco \utrio
                    \rmove(0 -90) \ltriz \ltriz \ltriz \rmove(0 10) \ltrio \ltrio \ltrio \rmove(0 10) \ltriz
                    \move(0 30) \bsegment \lvec(0 -10) \rmove(10 20) \rlvec(0 -20) \vtext(5, -8){...} \esegment
                    \vtext(7 12){$=$} \htext(3 3){$Y_0$}
                \end{texdraw}}
            \end{aligned}
        \end{equation*}
    \end{center}
    be a \Yw, where $G$ is the basic ground-state wall of weight $\lambda$. Then it is lengthy but straightforward to calculate the signatures of $Y_2$, $Y_1$ and $Y_0$:
    \begin{equation*}
        \begin{aligned}
            & 0\text{-signature of }Y_2 = +, \\
            & 0\text{-signature of }Y_1 = --++++, \\
            & 0\text{-signature of }Y_0 = +++, \\
            & 1\text{-signature of }Y_2 = +, \\
            & 1\text{-signature of }Y_1 = -, \\
            & 1\text{-signature of }Y_0 = --+.
        \end{aligned}
    \end{equation*}

    Let us verify some of the above calculations. We first focus on the 0-signature of $Y_1$.
    Remove two 0-blocks from $Y_1$ to get
    \begin{center}
        \begin{equation*}
            Y_1^{'} =  \hpic \raisebox{-0.8\height}{\begin{texdraw}
                \drawdim em \setunitscale 0.15 \linewd 0.4
                \move(0 10) \rlvec(0 25) \rlvec(10 0) \rlvec(0 -15) \rlvec(-10 -10) \lfill f:0.8
                \move(0 10) \utriz \utriz \reco \reco \utrio \utrio \utrio
                \rmove(0 -50) \ltrio \rmove(0 10) \ltriz
                \move(0 10) \bsegment \lvec(0 -10) \rmove(10 20) \rlvec(0 -20) \vtext(5, -8){...} \esegment
            \end{texdraw}}\ .
        \end{equation*}
    \end{center}
    Then we have
    \begin{equation*}
        \begin{array}{llll}
            s_2 = 0,        & t_2 = 0,      & \bar{s}_2 = 2,        & \bar{t}_2 = 2, \\
            s_1^{'} = 1,    & t_1^{'} = 4,  & \bar{s_1^{'}} = 5,    & \bar{t_1^{'}} = 8, \\
            s_0 = 7,        & t_0 = 8,      & \bar{s}_0 = 5,        & \bar{t}_0 = 6,
        \end{array}
    \end{equation*}
    which implies
$$H(Y_{2} \otimes Y_{1}')=0, \quad H(Y_{1}' \otimes Y_{0})=0.$$
It follows that $Y' = G \ot Y_2 \ot Y_1^{'} \ot Y_0$ is a \Yw.

    However, if we remove three 0-blocks from $Y_1$, then we get
    \begin{center}
        \begin{equation*}
            Y_1^{''} = \hpic \raisebox{-0.8\height}{\begin{texdraw}
                \drawdim em \setunitscale 0.15 \linewd 0.4
                \move(0 10) \rlvec(0 25) \rlvec(10 0) \rlvec(0 -15) \rlvec(-10 -10) \lfill f:0.8
                \move(0 10) \utriz \utriz \reco \reco \utrio \utrio \utrio
                \rmove(0 -50) \ltrio
                \move(0 10) \bsegment \lvec(0 -10) \rmove(10 20) \rlvec(0 -20) \vtext(5, -8){...} \esegment
            \end{texdraw}}
        \end{equation*}
    \end{center}
    and
    \begin{equation*}
        \begin{array}{llll}
            s_2 = 0,        & t_2 = 0,      & \bar{s}_2 = 2,        & \bar{t}_2 = 2, \\
            s_1^{''} = 0,   & t_1^{''} = 4, & \bar{s_1^{''}} = 4,   & \bar{t_1^{''}} = 8, \\
            s_0 = 7,        & t_0 = 8,      & \bar{s}_0 = 5,        & \bar{t}_0 = 6.
        \end{array}
    \end{equation*}
    Thus
    \begin{equation*}
 H(Y_{2} \otimes Y_1^{''}) = -2,
    \end{equation*}
    which implies $Y^{''} = G \ot Y_2 \ot Y_1^{''} \ot Y_0$ is \textit{not} a \Yw.
    Therefore, we can remove at most two 0-blocks from $Y_1$ and there are two $-$'s in the 0-signature of $Y_1$.

    Now we consider the 1-signature of $Y_0$.
    Add a pair of 1-blocks on $Y_0$ successively to obtain
    \begin{center}
        $Y_0^{'} = $ \hpic \raisebox{-0.8\height}{\begin{texdraw}
            \drawdim em \setunitscale 0.15 \linewd 0.4
            \move(0 10) \rlvec(0 15) \rlvec(10 10) \rlvec(0 -25) \rlvec(-10 0) \lfill f:0.8
            \move(0 10) \reco \utrio \utrio \utrio \recz \recz \utriz \utriz \utriz \reco \reco \utrio \utrio \utrio
            \rmove(0 -110) \ltriz \ltriz \ltriz \rmove(0 10) \ltrio \ltrio \ltrio \rmove(0 10) \ltriz
            \move(0 10) \bsegment \lvec(0 -10) \rmove(10 20) \rlvec(0 -20) \vtext(5, -8){...} \esegment
        \end{texdraw}} \ .
    \end{center}
    Then
    \begin{equation*}
        \begin{array}{llll}
            s_1 = 3,        & t_1 = 4,      & \bar{s}_1 = 7,        & \bar{t}_1 = 8, \\
            s_0^{'} = 7,    & t_0^{'} = 10, & \bar{s_0^{'}} = 5,    & \bar{t_0^{'}} = 8,
        \end{array}
    \end{equation*}
    and hence
 \begin{equation*}
H(Y_{1} \otimes Y_{0}')=0.
\end{equation*}
Therefore $Y^{'} = G \ot Y_2 \ot Y_1 \ot Y_0^{'}$ is a \Yw. However,
there is no place to add more 1-blocks. Thus we conclude there is
one $+$ in the 1-signature of $Y_0$.

    Similarly, one can verify the other signatures.

    Therefore, after canceling out all $(+,-)$-pairs, we see that
    the 0-signature of $Y$ is $(-+++++++)$ with $+$ in $Y_1$ and the 1-signature of $Y$ is $(--+)$ with $-$ in $Y_0$. Hence we obtain
    \begin{center}
        \begin{equation*}
            \begin{aligned}
                \tilde{F_0} Y & = \hpic \raisebox{-0.3\height}{\begin{texdraw}
                    \drawdim em \setunitscale 0.15 \linewd 0.4
                    \move(10 15) \rlvec(0 15) \rlvec(10 10) \rlvec(0 -10) \rlvec(10 10) \rlvec(0 -10) \rlvec(10 10) \rlvec(0 -5) \rlvec(10 0) \rlvec(0 -5) \rlvec(10 10) \rlvec(0 -25) \rlvec(-10 0) \rlvec(0 5) \rlvec(-10 -10) \rlvec(0 5) \rlvec(-30 0) \lfill f:0.8
                    \htext(0 22){$\cdots$}
                    \move(10 15) \rlvec(0 -15) \rmove(10 15) \rlvec(0 -15) \rmove(10 15) \rlvec(0 -15) \rmove(10 10) \rlvec(0 -10) \rmove(10 15) \rlvec(0 -15) \rmove(10 15) \rlvec(0 -15)
                    \move(10 15) \recz \utriz
                    \rmove(0 -10) \ltrio \ltrio
                    \rmove(10 -25) \recz \utriz
                    \rmove(0 -10) \ltrio \ltrio
                    \rmove(10 -25) \recz \utriz
                    \rmove(0 -10) \ltrio \ltrio
                    \rmove(10 -30) \utriz \utriz \reco \reco \utrio \utrio \utrio \grecz
                    \rmove(0 -55) \ltrio \rmove(0 10) \ltriz \ltriz \ltriz
                    \rmove(10 -55) \reco \utrio \utrio \utrio \recz \recz \utriz \utriz \utriz \reco \reco \utrio
                    \rmove(0 -90) \ltriz \ltriz \ltriz \rmove(0 10) \ltrio \ltrio \ltrio \rmove(0 10) \ltriz
                    \vtext(15 5){...} \vtext(25 5){...} \vtext(35 2){...} \vtext(45 2){...} \vtext(55 5){...}
                    \lpatt(0.3 1) \move(45 72) \lellip rx:9 ry:6 \lpatt()
                \end{texdraw}} \hpic , \\ \\
                \tilde{E_1} Y & = \hpic \raisebox{-0.41\height}{\begin{texdraw}
                    \drawdim em \setunitscale 0.15 \linewd 0.4
                    \move(10 15) \rlvec(0 15) \rlvec(10 10) \rlvec(0 -10) \rlvec(10 10) \rlvec(0 -10) \rlvec(10 10) \rlvec(0 -5) \rlvec(20 0) \rlvec(0 -15) \rlvec(-10 -10) \rlvec(0 10) \rlvec(-10 -10) \rlvec(0 5) \rlvec(-30 0) \lfill f:0.8
                    \htext(0 22){$\cdots$}
                    \move(10 15) \rlvec(0 -15) \rmove(10 15) \rlvec(0 -15) \rmove(10 15) \rlvec(0 -15) \rmove(10 10) \rlvec(0 -10) \rmove(10 10) \rlvec(0 -10) \rmove(10 20) \rlvec(0 -20)
                    \move(10 15) \recz \utriz
                    \rmove(0 -10) \ltrio \ltrio
                    \rmove(10 -25) \recz \utriz
                    \rmove(0 -10) \ltrio \ltrio
                    \rmove(10 -25) \recz \utriz
                    \rmove(0 -10) \ltrio \ltrio
                    \rmove(10 -30) \utriz \utriz \reco \reco \utrio \utrio \utrio
                    \rmove(0 -50) \ltrio \rmove(0 10) \ltriz \ltriz \ltriz
                    \rmove(10 -60) \utriz \utriz \reco \reco \utrio \utrio \utrio \recz \recz \utriz \utriz
                    \rmove(0 -80) \ltrio \rmove(0 10) \ltriz \ltriz \ltriz \rmove(0 10) \ltrio \ltrio
                    \vtext(15 5){...} \vtext(25 5){...} \vtext(35 2){...} \vtext(45 2){...} \vtext(55 5){...}
                \end{texdraw}} \hpic .
            \end{aligned}
        \end{equation*}
    \end{center}
    Note that we use the identification (\ref{eqn: identification}) when we compute $\tilde{E_1} Y$.
\end{ex}

For a \Yw\ $Y$, we define
\begin{equation*}
    \wt(Y) = \lambda - k_0 - \frac{1}{2} k_1 \alpha_1,
\end{equation*}
where $k_i$ is the number of $i$-blocks above the ground-state wall of $Y$ and let $\veps_i (Y)$ (resp. $\vphi_i (Y)$) be the number of $-$'s (resp. $+$'s) in the $i$-signature of $Y$ ($i = 0, 1$).

Recall the isomorphism $\Psi: \C \to \BBaff_{\leq 0}$ defined in
\eqref{eq:Baffine}, which induces an injective morphism
\begin{equation*}
    \Psi: \Y (\lambda) \cup \set{0} \to (\BBaff_{\leq 0})^{\ot \infty} \cup \set{0}
\end{equation*}
such that
\begin{equation}
    \begin{aligned}
        \begin{array}{rcl}
            0                       & \longmapsto   & 0 \\
            Y = (Y_k)_{k \geq 0}    & \longmapsto   & \cdots \ot \Psi(Y_k) \ot \cdots \ot \Psi(Y_1) \ot \Psi(Y_0).
        \end{array}
    \end{aligned}
\end{equation}

The following lemma is one of key ingredients of our construction.

\begin{lem} \label{Lemma 4.3}
    Let $Y = (Y_k)_{k \geq 0}$ be a \Yw. Then for $i=0,1$, we have
    \begin{enumerate}
        \item[(a)] $\vphi_i (Y_{k+1}) - \veps_i (Y_k) = \vphi_i (\Psi (Y_{k+1})) - \veps_i (\Psi (Y_k))$ for all $k \geq 0$,
        \item[(b)] $\Psi (\tilde{E}_i Y) = \tilde{e}_i \Psi(Y)$ if $\tilde{E}_i Y \neq 0$,
        \item[(c)] $\Psi (\tilde{F}_i Y) = \tilde{f}_i \Psi(Y)$ if $\tilde{F}_i Y \neq 0$.
    \end{enumerate}
\end{lem}
\begin{proof}
    To simplify the notation, we write $u_k = \Psi(Y_k)$ for $k \geq 0$. Recall that the morphism $\Psi: \C \to \BBaff_{\leq 0}$ commutes with $\tilde{e}_i$, $\tilde{f}_i$ ($i=0,1$) whenever all the maps involved send non-zero vectors to non-zero vectors. Hence
    \begin{equation*}
        \vphi_i (Y_k) \leq \vphi_i (u_k), \ \ \veps_i (Y_k) \leq \veps_i (u_k) \ \ \text{for all } k \geq 0.
    \end{equation*}
    Moreover, one can
easily verify that
    \begin{equation*}
        \begin{aligned}
            H (Y_{k+1} \ot \tilde{f}_i Y_k) & \geq H (Y_{k+1} \ot Y_k) \\
            H (\tilde{e}_i Y_{k+1} \ot Y_k) & \geq H (Y_{k+1} \ot Y_k) \\
        \end{aligned}
    \end{equation*}
    for all $k \geq 0$. It follows that
    \begin{equation} \label{eqn: Lemma_1}
        \begin{aligned}
            \vphi_i (Y_{k+1})   & = \max \set{0 \leq l \leq \vphi_i (u_{k+1}) ~|~ H(\tilde{f}_i^l Y_{k+1} \ot Y_k) \geq 0}, \\
            \veps_i (Y_{k})     & = \max \set{0 \leq l \leq \veps_i (u_k) ~|~ H(Y_{k+1} \ot \tilde{e}_i^l Y_k) \geq 0}.
        \end{aligned}
    \end{equation}

    Suppose $\vphi_i (u_{k+1}) > \veps_i (u_k)$. Then, for $0 \leq l \leq \veps_i (u_k)$, using the tensor product rule, we have
    \begin{equation*}
        \begin{aligned}
            u_{k+1} \ot \tilde{e}_i^{l} u_k & \stackrel{i}{\longrightarrow} \cdots \stackrel{i}{\longrightarrow} \tilde{f}_i^{\vphi_i (u_{k+1}) - \veps_i (u_k) + l} u_{k+1} \ot \tilde{e}_i^l u_k \\
            & \longrightarrow \cdots \longrightarrow \tilde{f}_i^{\vphi_i (u_{k+1}) - \veps_i (u_k) + l} u_{k+1} \ot \tilde{f}_i^l \tilde{e}_i^l u_k \\
            & = \tilde{f}_i^{\vphi_i (u_{k+1}) - \veps_i (u_k) + l} u_{k+1} \ot u_k.
        \end{aligned}
    \end{equation*}
    Therefore we obtain
    \begin{equation} \label{eqn: Lemma_2}
        H (u_{k+1} \ot \tilde{e}_i^l u_k) = H (\tilde{f}_i^{\vphi_i (u_{k+1}) - \veps_i (u_k) + l} u_{k+1} \ot u_k)
    \end{equation}
    for $0 \leq l \leq \veps_i (u_k)$.

    Set $l = \veps_i (u_k)$. Then (\ref{eqn: Lemma_1}) and (\ref{eqn: Lemma_2}) yield
    \begin{equation*}
        0 \leq H (Y_{k+1} \ot \tilde{e}_i^{\veps_i (Y_k)} Y_k) = H (\tilde{f}_i^{\vphi_i (u_{k+1}) - \veps_i (u_k) + \veps_i (Y_k)} Y_{k+1} \ot Y_k),
    \end{equation*}
    from which we obtain
    \begin{equation} \label{eqn: Lemma_3}
        \vphi_i (u_{k+1}) - \veps_i (u_k) + \veps_i (Y_k) \leq \vphi_i (Y_{k+1}).
    \end{equation}

    If $\veps_i (Y_k) = \veps_i (u_k)$, then (\ref{eqn: Lemma_3}) gives $\vphi_i (u_{k+1}) \leq \vphi_i (Y_{k+1})$. Thus $\vphi_i (Y_{k+1}) = \veps_i (u_{k+1})$ and we are done.

    If $\veps_i (Y_k) < \veps_i (u_k)$, set $l = \veps_i (Y_k) + 1$. Then by (\ref{eqn: Lemma_1}), we have
    \begin{equation*}
        0 > H (Y_{k+1} \ot \tilde{e}_i^{\veps_i (Y_k) + 1} Y_k) = H (\tilde{f}_i^{\vphi_i (u_{k+1}) - \veps_i (u_k) + \veps_i (Y_k) + 1} Y_{k+1} \ot Y_k).
    \end{equation*}
    Hence we have
    \begin{equation*}
        \vphi_i (u_{k+1}) - \veps_i (u_k) + \veps_i (Y_k) + 1 > \vphi_i (Y_{k+1}),
    \end{equation*}
    which implies
    \begin{equation*}
        \vphi_i (u_{k+1}) - \veps_i (u_k) + \veps_i (Y_k) \geq \vphi_i (Y_{k+1}).
    \end{equation*}
    Combined with (\ref{eqn: Lemma_3}), we conclude
    \begin{equation*}
        \vphi_i (Y_{k+1}) - \veps_i (Y_k) = \vphi_i (u_{k+1}) - \veps_i (u_k)
    \end{equation*}
    as desired.

    The other cases can be checked by a similar calculation.

\vskip 2mm

    (b) By (a), we have
    \begin{equation*}
        \tilde{F}_i Y = \cdots \ot Y_{k+1} \ot \tilde{f}_i Y_k \ot Y_{k-1} \ot \cdots \ot Y_0
    \end{equation*}
    if and only if
    \begin{equation*}
        \tilde{f}_i \Psi (Y) = \cdots \ot \Psi (Y_{k+1}) \ot \tilde{f}_i \Psi (Y_k) \ot \Psi (Y_{k-1}) \ot \cdots \ot \Psi (Y_0).
    \end{equation*}

    Therefore, for $i=0,1$, we have
    \begin{equation*}
        \begin{aligned}
            \Psi (\tilde{F}_i Y) & = \cdots \ot \Psi (Y_{k+1}) \ot \Psi (\tilde{f}_i Y_k) \ot \Psi (Y_{k-1}) \ot \cdots \ot \Psi (Y_0) \\
            & = \cdots \ot \Psi (Y_{k+1}) \ot \tilde{f}_i \Psi (Y_k) \ot \Psi (Y_{k-1}) \ot \cdots \ot \Psi (Y_0) \\
            & = \tilde{f}_i \Psi (Y).
        \end{aligned}
    \end{equation*}

    The statement (c) is an immediate consequence of (b).
\end{proof}

\begin{rem*}
    The argument of our proof applies to any case as long as
    we have a crystal isomorphism $\Psi: \C \to \BBaff_{\leq 0}$, where $\BB$ is a perfect crystal and $\C$ is the set of all columns of a \Yw\ model for $\BB$.
\end{rem*}

Now, let $\lambda = (l-2a) \Lambda_0 + a \Lambda_1$ be a dominant integral weight of level $l$. We first prove:
\begin{prop}
    The set $\Y (\lambda)$ of \Yws\ on $\lambda$ forms a \UAcry.
\end{prop}
\begin{proof}
    Let $Y = (Y_k)_{k \geq 0} = \langle s_k, \bar{s}_k, \bar{t}_k \rangle _{k \geq 0}$ be a \Yw.
    We have only to check
    \begin{equation*}
        \langle h_i, \wt(Y) \rangle = \vphi_i (Y) - \veps_i (Y) \ \text{ for } i=0, 1.
    \end{equation*}
    Recall that
    \begin{equation*}
        \begin{aligned}
            \Psi (Y_k) & = \Psi (\langle s_k, \bar{s}_k, \bar{t}_k \rangle) = \begin{cases}
                (\bar{s}_k - \frac{1}{2} \bar{t}_k, \frac{1}{2} \bar{t}_k)(-s_k) & \text{if} \quad \bar{s}_k \geq \bar{t}_k, \\
                (l - \frac{1}{2} \bar{t}_k, l - \bar{s}_k + \frac{1}{2} \bar{t}_k)(-s_k) & \text{if} \quad \bar{s}_k < \bar{t}_k.
            \end{cases}
        \end{aligned}
    \end{equation*}
    Then
    \begin{equation*}
        \wt \Psi (Y_k) = 2(\bar{t}_k - \bar{s}_k) \Lambda_0 + 2(\bar{s}_k - \bar{t}_k) \Lambda_1 - s_k \delta = -s_k \alpha_0 - \frac{1}{2} t_k \alpha_1.
    \end{equation*}

    Let $N$ be the smallest non-negative integer such that $Y_k$ is the basic ground-state column for all $k \geq N$. Then
    \begin{equation*}
        \begin{aligned}
            \wt \Psi (Y) & = \lambda - \sum_{k=0}^{N-1} (s_k \alpha_0 + \frac{1}{2} t_k \alpha_1) \\
                & = \lambda - k_0 \alpha_0 - \frac{1}{2} k_1 \alpha_1 = \wt (Y),
        \end{aligned}
    \end{equation*}
    where $k_i$ is the number of $i$-blocks above the ground-state wall of $Y$. Hence by Lemma \ref{Lemma 4.3}, we have
    \begin{equation*}
        \begin{aligned}
            \langle h_i, \wt (Y) \rangle & = \langle h_i, \wt \Psi (Y) \rangle \\
                & = \langle h_i, \sum_{k \geq 0} \wt \Psi (Y_k) \rangle \\
                & = \sum_{k \geq 0} \langle h_i, \wt \Psi (Y_k) \rangle \\
                & = \sum_{k \geq 0} (\vphi_i (\Psi (Y_k)) - \veps_i (\Psi (Y_k))) \\
                & = \sum_{k \geq 0} (\vphi_i (Y_k) - \veps_i (Y_k)) = \vphi_i (Y) - \veps_i (Y),
        \end{aligned}
    \end{equation*}
    which proves our claim.
\end{proof}

Next, we will show that the set $\R (\lambda)$ of \rYws\ on $\lambda$ provides a realization of the \UAcry\ $\BB (\lambda)$.

\begin{thm}
    The injective morphism $\Psi: \Y (\lambda) \to (\BB^\textnormal{aff}_{\leq 0})^{\ot \infty}$ induces a \UAcry\ isomorphism
    $$\Psi: \R (\lambda) \stackrel{\sim}{\longrightarrow} \PP^\textnormal{aff} (\lambda, \mathbf{0}) \ \ \text{sending} \ \ Y_\lambda \mapsto \mathbf{b}_\lambda^\textnormal{aff},$$
    where $\mathbf{0} = (\cdots, 0, \cdots, 0)$, $Y_\lambda$ is the basic ground-state wall of weight $\lambda$ and
    \begin{equation*}
        \mathbf{b}_\lambda^\textnormal{aff} = (\cdots, (a,a)(0), \cdots, (a,a)(0))
    \end{equation*}
    is the affine ground-state path of weight $\lambda$ in $\PP^\textnormal{aff} (\lambda, \mathbf{0})$.

    In particular, we have a \UAcry\ isomorphism
    $$\R (\lambda) \stackrel{\sim}{\longrightarrow} \BB (\lambda) \ \ \text{sending} \ \  Y_\lambda \mapsto
    u_\lambda.$$
\end{thm}
\begin{proof}
    Let $Y = (Y_k)_{k \geq 0}$ be a \rYw\ on $\lambda$. Assume that $\tilde{F}_i Y \neq 0$ and $\tilde{F}_i$ acts on the column $Y_k$ by $\tilde{f}_i$ ($i=0,1$). That is, $\vphi_i (Y_{k+1}) \leq \veps (Y_k)$, $\vphi_i (Y_k) > \veps (Y_{k-1})$ and
    \begin{equation*}
        \tilde{F}_i Y = \cdots \ot Y_{k+1} \ot \tilde{f}_i Y_k \ot Y_{k-1} \ot \cdots Y_0.
    \end{equation*}
    Since $Y$ is reduced, we have
    \begin{equation*}
        \begin{aligned}
            H(Y_{k+1} \ot \tilde{f}_i Y_k) & = H(\tilde{f}_i (Y_{k+1} \ot Y_k)) \\
                & = H(Y_{k+1} \ot Y_k) = 0, \\
            H(\tilde{f}_i Y_k \ot Y_{k-1}) & = H(\tilde{f}_i (Y_k \ot Y_{k-1})) \\
                & = H(Y_k \ot Y_{k-1}) = 0.
        \end{aligned}
    \end{equation*}

    Hence $Y$ and $\tilde{F}_i Y$ belongs to the same connected component, which implies $\tilde{F}_i Y \in \R (\lambda)$.

    By a similar argument, one can show that if $\tilde{E}_i Y \neq 0$, then $\tilde{E}_i Y \in \R (\lambda)$.

    Since $\Psi$ maps $Y_\lambda$ to $\mathbf{b}^\text{aff}_\lambda$, by Lemma \ref{Lemma 4.3}, $\Psi$ induces a \UAcry\ isomorphism $\R (\lambda) \stackrel{\sim}{\longrightarrow} \PPaff (\lambda, \mathbf{0})$.
\end{proof}

\begin{ex}
    Let $\lambda = 2 \Lambda_0 + \Lambda_1$. We illustrate part of the \UAcry\ $\R
    (\lambda)$.
    \begin{center}
        \begin{texdraw}
            \drawdim em \setunitscale 0.15 \linewd 0.4 \arrowheadtype t:V \arrowheadsize l:3 w:3

            \move(0 15) \rlvec(0 15) \rlvec(10 10) \rlvec(0 -5) \rlvec(10 0) \rlvec(0 -5) \rlvec(10 10) \rlvec(0 -25) \rlvec(-10 0) \rlvec(0 5) \rlvec(-10 -10) \rlvec(0 5) \rlvec(-10 0) \lfill f:0.8
            \move(0 15) \rlvec(0 -15) \rmove(10 10) \rlvec(0 -10) \rmove(10 15) \rlvec(0 -15) \rmove(10 15) \rlvec(0 -15)
            \move(0 15) \recz \utriz \utriz
            \rmove(0 -20) \ltrio \ltrio
            \rmove(10 -30) \utriz \utriz \reco \reco \utrio \utrio \utrio \recz
            \rmove(0 -55) \ltrio \rmove(0 10) \ltriz \ltriz \ltriz
            \rmove(10 -55) \reco \utrio \utrio \utrio \recz \recz \utriz \utriz \utriz \reco \reco \utrio
            \rmove(0 -90) \ltriz \ltriz \ltriz \rmove(0 10) \ltrio \ltrio \ltrio \rmove(0 10) \ltriz
            \vtext(5 5){...} \vtext(15 2){...} \vtext(25 2){...}

            \move(140 15) \rlvec(0 15) \rlvec(10 10) \rlvec(0 -5) \rlvec(10 0) \rlvec(0 -5) \rlvec(10 10) \rlvec(0 -25) \rlvec(-10 0) \rlvec(0 5) \rlvec(-10 -10) \rlvec(0 5) \rlvec(-10 0) \lfill f:0.8
            \move(140 15) \rlvec(0 -15) \rmove(10 10) \rlvec(0 -10) \rmove(10 15) \rlvec(0 -15) \rmove(10 15) \rlvec(0 -15)
            \move(140 15) \recz \utriz \utriz
            \rmove(0 -20) \ltrio \ltrio
            \rmove(10 -30) \utriz \utriz \reco \reco \utrio \utrio \utrio
            \rmove(0 -50) \ltrio \rmove(0 10) \ltriz \ltriz \ltriz
            \rmove(10 -55) \reco \utrio \utrio \utrio \recz \recz \utriz \utriz \utriz \reco \reco \utrio \utrio \utrio
            \rmove(0 -110) \ltriz \ltriz \ltriz \rmove(0 10) \ltrio \ltrio \ltrio \rmove(0 10) \ltriz
            \vtext(145 5){...} \vtext(155 2){...} \vtext(165 2){...}

            \lpatt(0.3 1) \move(15 72) \lellip rx:9 ry:6 \move(165 120) \lellip rx:13 ry:14 \lpatt()

            \move(70 75) \rlvec(0 15) \rlvec(10 10) \rlvec(0 -5) \rlvec(10 0) \rlvec(0 -5) \rlvec(10 10) \rlvec(0 -25) \rlvec(-10 0) \rlvec(0 5) \rlvec(-10 -10) \rlvec(0 5) \rlvec(-10 0) \lfill f:0.8
            \move(70 75) \rlvec(0 -15) \rmove(10 10) \rlvec(0 -10) \rmove(10 15) \rlvec(0 -15) \rmove(10 15) \rlvec(0 -15)
            \move(70 75) \recz \utriz \utriz
            \rmove(0 -20) \ltrio \ltrio
            \rmove(10 -30) \utriz \utriz \reco \reco \utrio \utrio \utrio
            \rmove(0 -50) \ltrio \rmove(0 10) \ltriz \ltriz \ltriz
            \rmove(10 -55) \reco \utrio \utrio \utrio \recz \recz \utriz \utriz \utriz \reco \reco \utrio
            \rmove(0 -90) \ltriz \ltriz \ltriz \rmove(0 10) \ltrio \ltrio \ltrio \rmove(0 10) \ltriz
            \vtext(75 65){...} \vtext(85 62){...} \vtext(95 62){...}
            \move(57 85) \ravec(-14 -8) \htext(48 84){0}
            \move(114 85) \ravec(14 -8) \htext(121 84){1}

            \move(20 200) \rlvec(0 15) \rlvec(10 10) \rlvec(0 -25) \rlvec(-10 0) \lfill f:0.8
            \move(20 200) \recz \utriz \utriz
            \rmove(0 -20) \ltrio \ltrio
            \move(20 200) \bsegment \lvec(0 -10) \rmove(10 10) \rlvec(0 -10) \vtext(5, -8){...} \esegment

            \move(140 200) \rlvec(0 15) \rlvec(10 10) \rlvec(0 -25) \rlvec(-10 0) \lfill f:0.8
            \move(140 200) \reco \utrio \utrio \utrio
            \rmove(0 -30) \ltriz \ltriz
            \move(140 200) \bsegment \lvec(0 -10) \rmove(10 10) \rlvec(0 -10) \vtext(5, -8){...} \esegment

            \lpatt(0.3 1) \move(23 223) \lellip rx:9 ry:7 \move(145 225) \lellip rx:13 ry:14 \lpatt()

            \move(37 185) \ravec(7 -14) \rtext td:-60 (46 165){$\cdots$}
            \move(134 185) \ravec(-7 -14) \rtext td:-120 (122 165){$\cdots$}

            \htext(72 247){$Y_\lambda = $} \move(88 240) \rlvec(0 15) \rlvec(10 10) \rlvec(0 -25) \rlvec(-10 0) \lfill f:0.8
            \move(88 240) \recz \utriz
            \rmove(0 -10) \ltrio \ltrio
            \move(88 240) \bsegment \lvec(0 -10) \rmove(10 10) \rlvec(0 -10) \vtext(5, -8){...} \esegment
            \move(57 245) \ravec(-14 -8) \htext(48 243){0}
            \move(114 245) \ravec(14 -8) \htext(121 243){1}
        \end{texdraw}
    \end{center}
\end{ex}

\vskip 5mm

\section{Young walls for $\BB (\infty)$}

\vskip 2mm

Let $\BB = \Z \times \Z$. We define a \UApcry\ structure on $\BB$ by
\begin{equation}
    \begin{aligned}
        \wt (x,y) & = 2(y-x) \Lambda_0 + (x-y) \Lambda_1, \\
        \veps_1 (x,y) & = y, \quad \vphi_1 (x,y) = x, \\
        \veps_0 (x,y) & = -2y + \abs{x-y}, \\
        \vphi_0 (x,y) & = -2x + \abs{x-y}, \\
        \tilde{e}_1 (x,y) & = (x+1, y-1), \\
        \tilde{f}_1 (x,y) & = (x-1, y+1), \\
        \tilde{e}_0 (x,y) & = \begin{cases}
            (x-1, y) & \text{if} \quad x > y, \\
            (x, y+1) & \text{if} \quad x \leq y,
        \end{cases} \\
        \tilde{f}_0 (x,y) & = \begin{cases}
            (x+1, y) & \text{if} \quad x \geq y, \\
            (x, y-1) & \text{if} \quad x < y.
        \end{cases}
    \end{aligned}
\end{equation}

Let $\BB_l$ be the level $l$ adjoint crystal. For $\lambda = (l-2a) \Lambda_0 + a \Lambda_1$ ($0 \leq a \leq \lfloor \frac{l}{2} \rfloor$), let $T_\lambda = \set{t_\lambda}$ be the crystal with the maps
\begin{equation*}
    \begin{aligned}
        & \wt (t_\lambda) = \lambda, \quad \veps_i (t_\lambda) = \vphi_i (t_\lambda) = -\infty, \\
        & \tilde{e}_i (t_\lambda) = \tilde{f}_i (t_\lambda) = 0 \ \text{ for } i=0,1.
    \end{aligned}
\end{equation*}

As in \cite{KKM}, $\set{\BB_l ~|~ l \geq 1}$ is a \textit{coherent family} with the limit $\BB$, where the inductive system is given by
\begin{equation*}
    \begin{array}{lrl}
        f_{l, (a,a)}:   & T_\lambda \ot \BB_l \ot T_{-\lambda}  & \longrightarrow \BB \\
                        & t_\lambda \ot (x,y) \ot t_{-\lambda}  & \longmapsto (x-a, y-a).
    \end{array}
\end{equation*}
Thus by the same argument given in \cite{KKM}, we have the path realization
\begin{equation}
    \BB (\infty) \stackrel{\sim}{\longrightarrow} \PP (\infty) \text{ as \UApcrys,}
\end{equation}
where $\PP (\infty)$ is the set of \textit{$\infty$-paths} $\mathbf{p} = (x_k, y_k)_{k \geq 0}$ such that
\begin{enumerate}
    \item $x_k, y_k \in \Z$ for all $k \geq 0$,
    \item $(x_k, y_k) = (0,0)$ for $k \gg 0$.
\end{enumerate}

As is the case with $B_{\text{ad}}$, we define a classical energy
function $h: \BB \ot \BB \to \Z$ by the formula \eqref{eq:h1} and
let $H: \BBaff \ot \BBaff \to \Z$ be the affine energy function
associated with $h$. Then we have a \UAcry\ isomorphism
\begin{equation}
    \BB (\infty) \stackrel{\sim}{\longrightarrow} \PPaff (\infty),
\end{equation}
where $\PPaff (\infty)$ is the set of all \textit{affine $\infty$-paths} $\mathbf{p} = ((x_k, y_k)(-m_k))_{k \geq 0}$ such that
\begin{enumerate}
    \item $x_k, y_k \in \Z$, $m_k \in \Z_{\geq 0}$ for all $k \geq 0$,
    \item $(x_k, y_k)(-m_k) = (0,0)(0)$ for $k \gg 0$,
    \item $H((x_{k+1}, y_{k+1})(-m_{k+1}) \ot (x_k, y_k)(-m_k)) = 0$ for all $k \geq 0$.
\end{enumerate}

Now we construct a \Yw\ model for the (level $\infty$) adjoint crystal $\BB$. We only use the colored blocks of half-unit thickness:
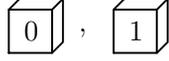
\begin{center}
    \vpic \raisebox{-0.4\height}{
        \begin{texdraw}
            \drawdim em \setunitscale 0.15  \linewd 0.4 \move(0 0)\lvec(10 0)\lvec(10 10)\lvec(0 10)\lvec(0 0) \move(10 0)\lvec(12.5 2.5)\lvec(12.5 12.5)\lvec(2.5 12.5)\lvec(0 10) \move(10 10)\lvec(12.5 12.5)  \htext(3.5 3){0}
        \end{texdraw}} \hskip 1mm , \hskip 1mm
    \raisebox{-0.4\height}{
        \begin{texdraw}
            \drawdim em \setunitscale 0.15  \linewd 0.4 \move(0 0)\lvec(10 0)\lvec(10 10)\lvec(0 10)\lvec(0 0) \move(10 0)\lvec(12.5 2.5)\lvec(12.5 12.5)\lvec(2.5 12.5)\lvec(0 10) \move(10 10)\lvec(12.5 12.5)  \htext(3.5 3){1}
        \end{texdraw}}
\end{center}
\vpic

The rules and patterns for the \Yw\ model is explained below.
\begin{enumerate}
    \item[(1)] The colored blocks will be stacked in the following way.
    \begin{center}
        \raisebox{-0.4\height}{\begin{texdraw}
            \drawdim em \setunitscale 0.15 \linewd 0.4
            \move(5 10) \utriz \utriz \utriz \utriz \utriz
            \rmove(0 -50) \ltrio \ltrio \ltrio \ltrio \ltrio
            \move(5 10) \bsegment \lvec(0 -10) \rmove(10 10) \rlvec(0 -10) \vtext(5, -8){...} \esegment
            \move(5 60) \bsegment \lvec(0 10) \rmove(10 -10) \rlvec(0 10) \vtext(5, 2){...} \esegment
            \move(0 30) \rlvec(20 0)
        \end{texdraw}} \hskip 3mm or \hskip 3mm \raisebox{-0.4\height}{\begin{texdraw}
            \drawdim em \setunitscale 0.15 \linewd 0.4
            \move(5 10) \utrio \utrio \utrio \utrio \utrio
            \rmove(0 -50) \ltriz \ltriz \ltriz \ltriz \ltriz
            \move(5 10) \bsegment \lvec(0 -10) \rmove(10 10) \rlvec(0 -10) \vtext(5, -8){...} \esegment
            \move(5 60) \bsegment \lvec(0 10) \rmove(10 -10) \rlvec(0 10) \vtext(5, 2){...} \esegment
            \move(0 30) \rlvec(20 0)
        \end{texdraw}}\hskip 3mm .
    \end{center}
    The horizontal lines are called the \textit{standard lines}.
    \item[(2)] We identify the following columns:
        \begin{equation} \label{col: Chapter 5}
            \begin{aligned}
                \raisebox{-0.7\height}{\begin{texdraw}
                    \drawdim em \setunitscale 0.15 \linewd 0.4
                    \move(5 10) \utriz \utriz \rlvec(0 20) \utriz \utriz \rmove(0 -60) \ltrio \ltrio \rmove(0 20) \ltrio \ltrio
                    \rmove(10 -40) \rlvec(0 20)
                    \rmove(-15 -30) \rlvec(20 0)
                    \move(5 10) \bsegment \lvec(0 -10) \rmove(10 10) \rlvec(0 -10) \vtext(5, -8){...} \esegment
                    \move(10 45) \bsegment \vtext(0, -8){...} \esegment
                    \move(16 20) \clvec(19 20)(19 23)(19 23) \rlvec(0 20)
                    \move(21 45) \clvec(19 45)(19 47)(19 47)
                    \move(21 45) \clvec(19 45)(19 43)(19 43)
                    \move(16 70) \clvec(19 70)(19 67)(19 67) \rlvec (0 -20)
                    \htext(23 43){$k$}
                \end{texdraw}} \ \ \ & = & \raisebox{-0.7\height}{\begin{texdraw}
                    \drawdim em \setunitscale 0.15 \linewd 0.4
                    \move(5 10) \utrio \utrio \rmove(0 -20) \ltriz \ltriz
                    \rlvec(0 40) \rmove(10 -40) \rlvec(0 40) \rmove(-10 -30) \rlvec(10 0) \rmove(-10 10) \rlvec(10 0) \rmove(-10 10) \rlvec(10 0)
                    \rmove(-15 10) \rlvec(20 0)
                    \move(5 10) \bsegment \lvec(0 -10) \rmove(10 10) \rlvec(0 -10) \vtext(5, -8){...} \esegment
                    \move(10 50) \bsegment \vtext(0, -8){...} \esegment
                    \move(16 30) \clvec(19 30)(19 33)(19 33) \rlvec(0 15)
                    \move(21 50) \clvec(19 50)(19 52)(19 52)
                    \move(21 50) \clvec(19 50)(19 48)(19 48)
                    \move(16 70) \clvec(19 70)(19 67)(19 67) \rlvec (0 -15)
                    \htext(23 48){$k$}
                \end{texdraw}} \\ \\
                \raisebox{-0.7\height}{\begin{texdraw}
                    \drawdim em \setunitscale 0.15 \linewd 0.4
                    \move(5 10) \utrio \utrio \rlvec(0 20) \utrio \rmove(0 -50) \ltriz \ltriz \rmove(0 20) \ltriz
                    \rmove(10 -30) \rlvec(0 20)
                    \rmove(-15 -30) \rlvec(20 0)
                    \move(5 10) \bsegment \lvec(0 -10) \rmove(10 10) \rlvec(0 -10) \vtext(5, -8){...} \esegment
                    \move(10 45) \bsegment \vtext(0, -8){...} \esegment
                    \move(16 20) \clvec(19 20)(19 23)(19 23) \rlvec(0 15)
                    \move(21 40) \clvec(19 40)(19 42)(19 42)
                    \move(21 40) \clvec(19 40)(19 38)(19 38)
                    \move(16 60) \clvec(19 60)(19 57)(19 57) \rlvec (0 -15)
                    \htext(23 38){$k$}
                \end{texdraw}} \ \ \ & = & \raisebox{-0.7\height}{\begin{texdraw}
                    \drawdim em \setunitscale 0.15 \linewd 0.4
                    \move(5 10) \utriz \utriz \rmove(0 -20) \ltrio \ltrio
                    \rlvec(0 30) \rmove(10 -30) \rlvec(0 30) \rmove(-10 -20) \rlvec(10 0) \rmove(-10 10) \rlvec(10 0) \rmove(-15 10) \rlvec(20 0)
                    \move(5 10) \bsegment \lvec(0 -10) \rmove(10 10) \rlvec(0 -10) \vtext(5, -8){...} \esegment
                    \move(10 50) \bsegment \vtext(0, -8){...} \esegment
                    \move(16 30) \clvec(19 30)(19 33)(19 33) \rlvec(0 10)
                    \move(21 45) \clvec(19 45)(19 47)(19 47)
                    \move(21 45) \clvec(19 45)(19 43)(19 43)
                    \move(16 60) \clvec(19 60)(19 58)(19 58) \rlvec (0 -10)
                    \htext(23 43){$k$}
                \end{texdraw}}
            \end{aligned}
        \end{equation}
    \item[(3)] We stack 0-blocks one-by-one and we stack two 1-blocks successively taking the identification (\ref{col: Chapter 5}) into account.
    \item[(4)] Under the identification (\ref{col: Chapter 5}), we can stack the blocks in the front repeatedly, but \textit{not} in the back unless there are blocks in the front with the same height.
    \item[(5)] The \textit{basic ground-state column} is defined to be
        \begin{equation*}
            \raisebox{-0.4\height}{\begin{texdraw}
                \drawdim em \setunitscale 0.15 \linewd 0.4
                \move(5 10) \utriz \utriz
                \rmove(0 -20) \ltrio \ltrio
                \move(5 10) \bsegment \lvec(0 -10) \rmove(10 10) \rlvec(0 -10) \vtext(5, -8){...} \esegment
                \move(0 30) \rlvec(20 0)
            \end{texdraw}} \hskip 3mm = \hskip 3mm \raisebox{-0.4\height}{\begin{texdraw}
                \drawdim em \setunitscale 0.15 \linewd 0.4
                \move(5 10) \utrio \utrio
                \rmove(0 -20) \ltriz \ltriz
                \move(5 10) \bsegment \lvec(0 -10) \rmove(10 10) \rlvec(0 -10) \vtext(5, -8){...} \esegment
                \move(0 30) \rlvec(20 0)
            \end{texdraw}} \ .
        \end{equation*}
\end{enumerate}

A column obtained from the basic ground-state column by adding or removing some $\delta$-blocks is called a \textit{ground-state column}.

Let $G$ be a ground-state column and let $C$ be a column obtained by stacking finitely many blocks on $G$. Let $Z$ be the basic ground-state column. We define
\begin{equation}
    \begin{aligned}
        s       & = s(C) = \text{the number of 0-blocks in $C$ above $G$}, \\
        t       & = t(C) = \text{the number of 1-blocks in $C$ above $G$}, \\
        \bar{s} & = \bar{s}(C) = \begin{cases}
            \text{the number of 0-blocks in $C$ above $Z$ if $C$ lies above $Z$}, \\
            -(\text{the number of 0-blocks in $Z$ above $C$ if $Z$ lies above $C$}),
        \end{cases} \\
        \bar{t} & = \bar{t}(C) = \begin{cases}
            \text{the number of 1-blocks in $C$ above $Z$ if $C$ lies above $Z$}, \\
            -(\text{the number of 1-blocks in $Z$ above $C$ if $Z$ lies above $C$}).
        \end{cases}
    \end{aligned}
\end{equation}
\begin{ex} \hfill
    \begin{enumerate}
      \item Let $C = $ \hpic \raisebox{-0.7\height}{\begin{texdraw}
                \drawdim em \setunitscale 0.15 \linewd 0.4
                \move(5 10) \utrio \utrio \utrio \utrio \utrio \utrio
                \rmove(0 -60) \ltriz \ltriz \ltriz \ltriz
                \move(5 10) \bsegment \lvec(0 -10) \rmove(10 10) \rlvec(0 -10) \vtext(5, -8){...} \esegment
                \move(0 10) \rlvec(20 0) \htext(22 8){$Z$}
                \move(0 30) \rlvec(20 0) \htext(22 28){$G$}
            \end{texdraw}} \ . \vpic
            Then $s=2$, $t=4$, $\bar{s} = 4$, $\bar{t} = 6$. \vpic
      \item Let $C = $ \hpic \raisebox{-0.7\height}{\begin{texdraw}
                \drawdim em \setunitscale 0.15 \linewd 0.4
                \move(5 10) \utriz \utriz \utriz \utriz \utriz \utriz \utriz
                \rmove(0 -70) \ltrio \ltrio \ltrio \ltrio
                \move(5 10) \bsegment \lvec(0 -10) \rmove(10 10) \rlvec(0 -10) \vtext(5, -8){...} \esegment
                \move(0 10) \rlvec(20 0) \htext(22 8){$G$}
                \move(0 30) \rlvec(20 0) \htext(22 28){$Z$}
            \end{texdraw}} \ . \vpic
            Then $s=7$, $t=4$, $\bar{s} = 5$, $\bar{t} = 2$. \vpic
      \item Let $C = $ \hpic \raisebox{-0.7\height}{\begin{texdraw}
                \drawdim em \setunitscale 0.15 \linewd 0.4
                \move(5 5) \utrio \utrio \utrio \utrio
                \rmove(0 -40) \ltriz \ltriz \ltriz
                \rmove(0 10) \rlvec(10 10) \rlvec(-10 0) \rlvec(10 10) \rlvec(0 -30) \rmove(-10 30) \rlvec(0 -30)
                \move(0 5) \rlvec(20 0) \htext(22 3){$G$}
                \move(0 65) \rlvec(20 0) \htext(22 63){$Z$}
            \end{texdraw}} \ . \vpic
            Then $s=3$, $t=4$, $\bar{s} = -3$, $\bar{t} = -2$.
    \end{enumerate}
\end{ex}

Note that any column $C$ of our \Yw\ model is uniquely determined by $s$, $\bar{s}$, $\bar{t}$.

Let $\C$ be the set of all columns of our \Yw\ model and let
\begin{equation*}
    \BBaff_{\leq 0} = \set{(x,y)(-m) ~|~ x,y \in \Z, \ m \in \Z_{\geq 0}}.
\end{equation*}
Define a bijection $\Psi: \C \to \BBaff_{\leq 0}$ by
\begin{equation}
    \Psi (s, \bar{s}, \bar{t}) = \begin{cases}
        \left( \bar{s} - \frac{1}{2} \bar{t}, \frac{1}{2} \bar{t} \right) (-s)      & \text{if} \quad \bar{s} \geq \bar{t}, \\
        \left( -\frac{1}{2} \bar{t}, -\bar{s} + \frac{1}{2} \bar{t} \right) (-s)    & \text{if} \quad \bar{s} < \bar{t}.
    \end{cases}
\end{equation}
Then one can easily check that $\Psi$ commutes with $\tilde{e}_i$, $\tilde{f}_i$ ($i=0,1$) whenever all the maps involved send non-zero vectors to non-zero vectors.

Let $C = \langle s, t, \bar{s}, \bar{t} \rangle$, $C^{'} = \langle
s', t', \bar{s'}, \bar{t'} \rangle \in \C$. Using the isomorphism
$\Psi$, the affine energy function $H$ can be written as
$$H(\Psi(C) \otimes \Psi(C')) = -s + s' -h(C,C'),$$
where
\begin{equation} \label{eq:h1-b}
h(C,C')=\begin{cases} \text{max}\left(\begin{aligned}&
\bar{s}-\bar{s'}, \bar{s}+\bar{s'}-2 \bar{t'}, \\ &
\bar{s'}-\bar{s}, \bar{s'}-3 \bar{s} + 2 \bar{t}
\end{aligned}\right) & \ \ \text{if} \ \
\bar{s} \ge \bar{t}, \ \bar{s'} \ge \bar{t'}, \\
\text{max} \left(\begin{aligned}& \bar{s}+\bar{s'}, - 3 \bar{s}-
\bar{s'} + 2 \bar{t},  \\ &  -\bar{s}-\bar{s'},
   \bar{s}+3 \bar{s'} - 2 \bar{t'}\end{aligned}\right) & \ \
\text{if} \ \ \bar{s} \ge \bar{t}, \ \bar{s'} < \bar{t'},
\\
\text{max} \left(\begin{aligned} &  - \bar{s} - \bar{s'},  - \bar{s}
+
\bar{s'} + 2 \bar{t}, \\
& \bar{s} + \bar{s'}, - \bar{s} + \bar{s'} -2 \bar{t'}\end{aligned}
\right) &  \ \ \text{if} \ \bar{s}< \bar{t},
\bar{s'} \ge \bar{t'}, \\
\text{max} \left(\begin{aligned} & \bar{s} - \bar{s'}, - \bar{s} -
\bar{s'} + 2 \bar{t}, \\ & \bar{s'} - \bar{s}, -\bar{s} + 3 \bar{s'}
-2 \bar{t'} \end{aligned}\right) & \ \ \text{if} \ \bar{s} <
\bar{t}, \ \bar{s'} < \bar{t'}.
\end{cases}
\end{equation}

The notion of \textit{(basic) ground-state \Yws}, \textit{\Yws},
\textit{\rYws}, \textit{$i$-signatures} and \textit{Kashiwara
operators} are defined in the same way as in Section \ref{Sec:
Crystal structure}. Let $\Y (\infty)$ (resp. $\R (\infty)$) denote
the set of all \Yws\ (resp. \rYws). For a \Yw\ $Y = (Y_k)_{k \geq
0}$, we define
\begin{enumerate}
    \item $\wt (Y) = -k_0 \alpha_0 - \frac{1}{2} k_1 \alpha_1$, where $k_i$ is the number of $i$-blocks above the ground-state wall of $Y$ $(i=0,1)$,
    \item $\veps_i (Y)$ is the number of $-$'s in the $i$-signature of $Y$,
    \item $\vphi_i (Y)$ is the number of $+$'s in the $i$-signature of $Y$.
\end{enumerate}

Using the same argument in Section \ref{Sec: Crystal structure}, one can prove:

\begin{thm} \hfill

    {\rm (a)} The sets $\Y (\infty)$ and $\R (\infty)$ are \UAcrys.

    {\rm (b)} There are \UAcry\ isomorphisms
    \begin{equation*}
        \begin{array}{ccccc}
            \R (\infty) & \stackrel{\sim}{\longrightarrow}  & \PP^\textnormal{aff}(\infty) & \stackrel{\sim}{\longrightarrow}    & \BB (\infty) \\
            Y_\infty    & \longmapsto                       & \mathbf{b}^\textnormal{aff}_\infty  & \longmapsto           & \mathbf{1},
        \end{array}
    \end{equation*}
    where $Y_\infty$ is the basic ground-state wall, $\mathbf{b}^\textnormal{aff}_\infty$ is the affine ground-state path in $\PP^\textnormal{aff}(\infty)$ and $\mathbf{1}$ is the highest weight vector of $\BB (\infty)$.
\end{thm}

\begin{ex}
    We illustrate part of the \UAcry\ $\R (\infty)$.
        \begin{center}
        \begin{texdraw}
            \drawdim em \setunitscale 0.15 \linewd 0.4 \arrowheadtype t:V \arrowheadsize l:3 w:3

            \move(5 20) \rlvec(0 20) \rlvec(30 0) \rlvec(0 -30) \rlvec(-10 0) \rlvec(0 10) \rlvec(-20 0) \lfill f:0.8
            \rmove(20 -10) \setgray 1 \rlvec(10 0) \setgray 0
            \rmove(-30 10) \utriz \utriz \utriz \rmove(10 -30) \utriz \utriz \utriz \utriz \rmove(10 -40) \utrio \utrio \utrio \utrio \utrio \utrio
            \rmove(-20 -60) \ltrio \ltrio \rmove(10 -20) \ltrio \ltrio \ltrio \ltrio \rmove(10 -40) \ltriz \ltriz \ltriz \ltriz
            \move(0 40) \rlvec(30 0) \rmove(-10 40) \rlvec(20 0)
            \move(5 20) \bsegment \lvec(0 -10) \rmove(10 10) \rlvec(0 -10) \vtext(5, -8){...} \esegment
            \move(15 20) \bsegment \lvec(0 -10) \rmove(10 10) \rlvec(0 -10) \vtext(5, -8){...} \esegment
            \move(25 20) \bsegment \lvec(0 -10) \rmove(10 10) \rlvec(0 -10) \vtext(5, -8){...} \esegment

            \move(145 10) \rlvec(0 30) \rlvec(30 0) \rlvec(0 -20) \rlvec(-20 0) \rlvec(0 -10) \rlvec(-10 0) \lfill f:0.8
            \setgray 1 \rlvec(10 0) \setgray 0
            \rmove(-10 10) \utriz \utriz \rmove(10 -20) \utriz \utriz \utriz \utriz \rmove(10 -40) \utrio \utrio \utrio \utrio \utrio \utrio \utrio \utrio
            \rmove(-20 -80) \ltrio \ltrio \rmove(10 -20) \ltrio \ltrio \ltrio \ltrio \rmove(10 -40) \ltriz \ltriz \ltriz \ltriz
            \move(140 40) \rlvec(30 0) \rmove(-10 40) \rlvec(20 0)
            \move(145 20) \bsegment \lvec(0 -10) \rmove(10 10) \rlvec(0 -10) \vtext(5, -8){...} \esegment
            \move(155 20) \bsegment \lvec(0 -10) \rmove(10 10) \rlvec(0 -10) \vtext(5, -8){...} \esegment
            \move(165 20) \bsegment \lvec(0 -10) \rmove(10 10) \rlvec(0 -10) \vtext(5, -8){...} \esegment

            \move(20 10) \bsegment \vtext(0, -8){...} \esegment
            \move(160 10) \bsegment \vtext(0, -8){...} \esegment

            \move(75 60) \rlvec(0 20) \rlvec(30 0) \rlvec(0 -20) \rlvec(-30 0) \lfill f:0.8
            \rmove(0 0) \utriz \utriz \rmove(10 -20) \utriz \utriz \utriz \utriz \rmove(10 -40) \utrio \utrio \utrio \utrio \utrio \utrio
            \rmove(-20 -60) \ltrio \ltrio \rmove(10 -20) \ltrio \ltrio \ltrio \ltrio \rmove(10 -40) \ltriz \ltriz \ltriz \ltriz
            \move(70 80) \rlvec(30 0) \rmove(-10 40) \rlvec(20 0)
            \move(75 60) \bsegment \lvec(0 -10) \rmove(10 10) \rlvec(0 -10) \vtext(5, -8){...} \esegment
            \move(85 60) \bsegment \lvec(0 -10) \rmove(10 10) \rlvec(0 -10) \vtext(5, -8){...} \esegment
            \move(95 60) \bsegment \lvec(0 -10) \rmove(10 10) \rlvec(0 -10) \vtext(5, -8){...} \esegment
            \move(62 75) \ravec(-14 -8) \htext(53 76){0}
            \move(119 75) \ravec(14 -8) \htext(126 76){1}

            \move(70 140) \rlvec(0 30) \rlvec(10 0) \rlvec(0 -30) \rlvec(-10 0) \lfill f:0.8
            \setgray 1 \rlvec(10 0) \setgray 0
            \rmove(-10 10) \utrio \utrio \utrio \utrio \rmove(0 -40) \ltriz \ltriz \ltriz
            \rmove(0 -30) \bsegment \lvec(0 -10) \rmove(10 10) \rlvec(0 -10) \vtext(5, -8){...} \esegment
            \move(65 190) \rlvec(20 0)

            \move(100 140) \rlvec(0 30) \rlvec(10 0) \rlvec(0 -30) \rlvec(-10 0) \lfill f:0.8
            \setgray 1 \rlvec(10 0) \setgray 0
            \rmove(-10 10) \utrio \utrio \utrio \utrio \rmove(0 -40) \ltriz \ltriz \ltriz
            \rmove(0 -30) \bsegment \lvec(0 -10) \rmove(10 10) \rlvec(0 -10) \vtext(5, -8){...} \esegment
            \move(95 170) \rlvec(20 0)

            \move(90 135) \bsegment \vtext(0, -8){...} \esegment

            \move(25 190) \rlvec(0 30) \rlvec(10 0) \rlvec(0 -30) \rlvec(-10 0) \lfill f:0.8
            \setgray 1 \rlvec(10 0) \setgray 0
            \rmove(-10 10) \utriz \utriz \utriz \rmove(0 -30) \ltrio \ltrio
            \rmove(-5 0) \rlvec(20 0)
            \move(30 200) \bsegment \vtext(0, -8){...} \esegment

            \move(145 190) \rlvec(0 30) \rlvec(10 0) \rlvec(0 -30) \rlvec(-10 0) \lfill f:0.8
            \setgray 1 \rlvec(10 0) \setgray 0
            \rmove(-10 10) \utrio \utrio \utrio \utrio \rmove(0 -40) \ltriz \ltriz
            \rmove(-5 0) \rlvec(20 0)
            \move(150 200) \bsegment \vtext(0, -8){...} \esegment

            \move(43 197) \ravec(14 -8) \htext(49 199){1}
            \move(133 197) \ravec(-14 -8) \htext(125 199){0}

            \htext(72 248){$Y_\infty = $} \move(95 230) \rlvec(0 30) \rlvec(10 0) \rlvec(0 -30) \rlvec(-10 0) \lfill f:0.8
            \setgray 1 \rlvec(10 0) \setgray 0
            \rmove(-10 10) \utriz \utriz \rmove(0 -20) \ltrio \ltrio
            \rmove(-5 0) \rlvec(20 0)
            \move(100 240) \bsegment \vtext(0, -8){...} \esegment

            \move(57 235) \ravec(-14 -4) \htext(49 238){0}
            \move(119 235) \ravec(14 -4) \htext(125 238){1}
        \end{texdraw}
    \end{center}
\end{ex}

\end{document}